\documentclass[reqno]{amsart}
\usepackage{amsfonts}
\usepackage{amsmath}
\usepackage{amsthm, amscd}
\usepackage{mathtools}
\usepackage{amssymb}
\usepackage{mathrsfs}

\usepackage[alphabetic]{amsrefs}
\usepackage{etex}
\usepackage{hyperref}
\hypersetup{
  colorlinks = true,
  linkcolor  = black
}

\usepackage{tikz}
\usepackage{xypic}

\allowdisplaybreaks

\theoremstyle{plain}\newtheorem{Theorem}{Theorem}[section]
\theoremstyle{plain}
\theoremstyle{plain}\newtheorem{Lemma}[Theorem]{Lemma}
\theoremstyle{plain}\newtheorem{Definition}[Theorem]{Definition}
\theoremstyle{plain}\newtheorem{Proposition}[Theorem]{Proposition}
\theoremstyle{plain}
\theoremstyle{plain}\newtheorem{Question}[Theorem]{Question}
\theoremstyle{plain}\newtheorem*{Theorem*}{Theorem}

\makeatletter
\newtheorem*{rep@theorem}{\rep@title}
\newcommand{\newreptheorem}[2]{%
\newenvironment{rep#1}[1]{%
 \def\rep@title{#2 \ref{##1}}%
 \begin{rep@theorem}}%
 {\end{rep@theorem}}}
\makeatother
\theoremstyle{plain}\newreptheorem{theorem}{Theorem}

\theoremstyle{remark}\newtheorem{remark}[Theorem]{Remark}

\numberwithin{equation}{section}

\DeclareMathOperator{\rank}{rank}
\DeclareMathOperator{\II}{I}

\DeclareMathOperator{\AHI}{AHI}
\DeclareMathOperator{\AKh}{AKh}
\DeclareMathOperator{\muu}{\mu^{orb}}

\DeclareMathOperator{\SU}{SU}

\DeclareMathOperator{\pt}{pt}
\DeclareMathOperator{\id}{id}

\DeclareMathOperator{\CKh}{CKh}
\DeclareMathOperator{\SiHI}{\Sigma HI}
\DeclareMathOperator{\SiKh}{\Sigma Kh}

\newcommand{\bC}{\mathbb{C}}

\newcommand{\bZ}{\mathbb{Z}}

\newcommand{\bfu}{{\bf u}}
\newcommand{\bfv}{{\bf v}}
\newcommand{\bfw}{{\bf w}}

\author{Yi Xie, Boyu Zhang}
\title{Instantons and Khovanov skein homology on $I\times T^2$}

\begin{document}

\begin{abstract}
Asaeda, Przytycki and Sikora \cite{APS} defined a generalization of Khovanov homology for links in $I$--bundles over  compact surfaces. We prove that for a link $L\subset (-1,1)\times T^2$, the Asaeda-Przytycki-Sikora homology of $L$ has rank $2$ with $\bZ/2$--coefficients if and only if $L$ is isotopic to an embedded knot in $\{0\}\times T^2$.
\end{abstract}

\maketitle

\setcounter{tocdepth}{1}

\section{Introduction}
\label{sec_intro} 

 Khovanov \cite{Kh-Jones} defined bi-graded homology groups for  links in $S^3$, which are categorifications of the Jones polynomial. Kronheimer and Mrowka \cite{KM:Kh-unknot} proved that Khovanov homology detects the unknot. A generalization of Khovanov homology was introduced by Asaeda-Przytycki-Sikora \cite{APS} for links in $I$--bundles over compact surfaces (possibly with boundary). The Asaeda-Przytycki-Sikora homology is a categorification of the Kauffman bracket skein module, and we will call it the \emph{Khovanov skein homology}.

In the case that the compact surface is an annulus and the $I$--bundle is trivial, the Khovanov skein homology is also called the \emph{annular Khovanov homology}. In \cite{XZ:excision}, the authors proved that the annular Khovanov homology detects the trivial links and the closures of trivial braids. The current paper gives a detection result for Khovanov skein homology when the surface is a torus. We prove that for a link $L\subset (-1,1)\times T^2$, the Khovanov skein homology of $L$ with $\bZ/2$--coefficients has rank $2$ if and only if $L$ can be be isotoped to a knot embedded in $\{0\}\times T^2$.

We briefly review the Khovanov skein homology defined by \cite{APS} for oriented framed links in $(-1,1)\times T^2$. The Khovanov skein homology for links in $(-1,1)\times T^2$ is graded over $\bZ^2\oplus \bZ\mathcal{C}$, where $\mathcal{C}$ is the set of isotopy classes of un-oriented essential simple closed curves on $T^2$. The gradings over the $\bZ^2$--factor are usually called the homological grading and the quantum grading. We will ignore these two gradings in our discussion because they will not be used in this paper. The Khovanov skein homology as a $\bZ\mathcal{C}$--graded group does not depend on the orientation or the framing of $L$, therefore our discussion below will not make reference to the orientation or the framing.

Suppose $L$ is given by a diagram $D$ on $T^2$ with $d$ crossings, and fix an order of the crossings.
For $v=(v_1,v_2,\cdots, v_d)\in \{0,1\}^d$, resolving the crossings of $D$ by a sequence of 0-smoothings and 1-smoothings (see Figure \ref{01smoothing}) given by $v$ changes $D$ to a collection of circles on $T^2$.
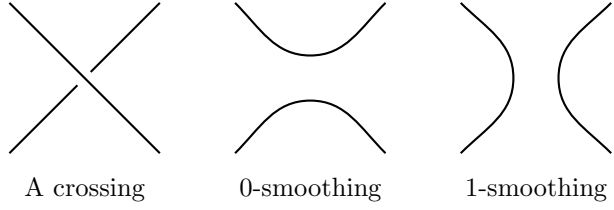
\begin{figure}
\begin{tikzpicture}
\draw[thick] (1,-1) to (-1,1); \draw[thick,dash pattern=on 1.3cm off 0.25cm] (1,1) to (-1,-1);  \node[below] at (0,-1.2) {A crossing};

\draw[thick] (2,1)  to [out=315,in=180]  (3,0.3) to [out=0,in=225]   (4,1);
\draw[thick] (2,-1)  to [out=45,in=180]  (3,-0.3) to [out=0,in=135]   (4,-1);  \node[below] at (3,-1.2) {0-smoothing};

\draw[thick] (5,1)  to [out=315,in=90]  (5.7,0) to [out=270,in=45]    (5,-1);  \node[below] at (6,-1.2) {1-smoothing};
\draw[thick] (7,1)  to [out=225,in=90]  (6.3,0) to [out=270,in=135]   (7,-1);
\end{tikzpicture}
\caption{Two types of smoothings}\label{01smoothing}
\end{figure}
There are two types of circles in the resolved diagram: circles that bound disks in $T^2$ and circles that are essential. We will call them the trivial circles and the essential circles respectively.

If $\gamma$ is a circle on $T^2$, we define a rank-2 
 free abelian group $V(\gamma):=\mathbb{Z}\{\mathbf{v}(\gamma)_+,\mathbf{v}(\gamma)_-\}$. The group $V(\gamma)$ is graded over $\bZ\mathcal{C}$ as follows. If $\gamma$ is a trivial circle, then $\deg \mathbf{v}(\gamma)_\pm=0$; if $\gamma$ is an essential circle, then $\deg \mathbf{v}(\gamma)_\pm= \pm[\gamma]$, where $[\gamma]\in \mathcal{C}$ is the isotopy class of $\gamma$. 
Given $v\in \{0,1\}^d$, let $S_v$ be the set of circles
in the corresponding resolved diagram, define
\begin{equation*}
  CKh_v(L)= \bigotimes_{\gamma\in S_v} V(\gamma).
\end{equation*}

If we change a 0-smoothing to a 1-smoothing, then the resolved diagram is changed by merging two circles into one circle, or splitting one circle into two circles.
Let $u\in \{0,1\}^d$ be the label of the new resolution, we define a map from $CKh_v(L)$ to $CKh_u(L)$ as follows. It is sufficient to specify the maps between the abelian groups associated to the
circles involved in the merging or splitting. In the case that two circles are merged into one, let $\gamma_1$ and $\gamma_2$ be the two circles before the merging and let $\gamma$ be the circle after the merging. The merge map is commutative on $\gamma_1$ and $\gamma_2$. If $\gamma_1$ and $\gamma_2$ are both trivial circles, the map is given by
\begin{align*}
  \mathbf{v}(\gamma_1)_+ \otimes \mathbf{v}(\gamma_2)_+ &\mapsto \mathbf{v}(\gamma)_+,  &\mathbf{v}(\gamma_1)_+& \otimes \mathbf{v}(\gamma_2)_- \mapsto \mathbf{v}(\gamma)_- ,\\
  \mathbf{v}(\gamma_1)_- \otimes \mathbf{v}(\gamma_2)_+ &\mapsto \mathbf{v}(\gamma)_-,  &\mathbf{v}(\gamma_1)_-& \otimes \mathbf{v}(\gamma_2)_- \mapsto 0.
\end{align*}
If $\gamma_1$  is trivial and $\gamma_2$ is essential, the map is given by
\begin{align*}
  \mathbf{v}(\gamma_1)_+\otimes \mathbf{v}(\gamma_2)_+ &\mapsto \mathbf{v}(\gamma)_+, &\mathbf{v}(\gamma_1)_+&\otimes \mathbf{v}(\gamma_2)_- \mapsto \mathbf{v}(\gamma)_- ,\\
  \mathbf{v}(\gamma_1)_-\otimes \mathbf{v}(\gamma_2)_+ &\mapsto 0,           &\mathbf{v}(\gamma_1)_-&\otimes \mathbf{v}(\gamma_2)_- \mapsto 0.
\end{align*}
If both $\gamma_1$ and $\gamma_2$ are essential, since they are disjoint circles on $T^2$, we have $\gamma_1$ and $\gamma_2$ are isotopic to each other, and $\gamma_3$ is a trivial circle. The merge map is given by 
\begin{align*}
  \mathbf{v}(\gamma_1)_+\otimes \mathbf{v}(\gamma_2)_+ &\mapsto 0,   &\mathbf{v}(\gamma_1)_+& \otimes \mathbf{v}(\gamma_2)_- \mapsto \mathbf{v}(\gamma)_-,\\
  \mathbf{v}(\gamma_1)_-\otimes \mathbf{v}(\gamma_2)_+ &\mapsto \mathbf{v}(\gamma)_-, &\mathbf{v}(\gamma_1)_-&\otimes \mathbf{v}(\gamma_2)_- \mapsto 0.
\end{align*}
In the case that one circle splits into two circles, let $\gamma$ be the circle that splits into $\gamma_1$ and $\gamma_2$. If $\gamma_1$ and $\gamma_2$ are both trivial circles, the map is given by
$$
\mathbf{v}(\gamma)_+ \mapsto \mathbf{v}(\gamma_1)_+\otimes \mathbf{v}(\gamma_2)_- + \mathbf{v}(\gamma_1)_-\otimes \mathbf{v}(\gamma_2)_+, \quad \mathbf{v}(\gamma)_- \mapsto \mathbf{v}(\gamma_1)_-\otimes \mathbf{v}(\gamma_2)_-.
$$
If $\gamma_1$  is trivial and $\gamma_2$ is essential, the map is given by
$$
\mathbf{v}(\gamma)_+ \mapsto \mathbf{v}(\gamma_1)_-\otimes \mathbf{v}(\gamma_2)_+, \quad \mathbf{v}(\gamma)_- \mapsto \mathbf{v}(\gamma_1)_-\otimes \mathbf{v}(\gamma_2)_-.
$$
If both $\gamma_1$ and $\gamma_2$ are essential, the map is given by
$$
\mathbf{v}(\gamma)_+ \mapsto \mathbf{v}(\gamma_1)_+\otimes \mathbf{v}(\gamma_2)_- + \mathbf{v}(\gamma_1)_-\otimes \mathbf{v}(\gamma_2)_+ , \quad \mathbf{v}(\gamma)_- \mapsto 0.
$$

Hence we have a map $d_{vu}: \CKh_v(L) \to \CKh_u(L)$ whenever $u$ is obtained from $v$ by changing one coordinate from 0 to 1. It is straightforward to verify that $d_{vu}$ preserves the gradings over $\bZ\mathcal{C}$. Let $e_i$ be the $i$--th standard basis vector of $\mathbb{R}^c$.
Define
\begin{equation*}
  \CKh(L):=\bigoplus_{v\in \{0,1\}} \CKh_v(L),
\end{equation*}
and define an endomorphisms of $\CKh(L)$ by
\begin{equation*}
  D:= \sum_i \sum_{u-v=e_i}  (-1)^{\sum_{i<j\le c}v_j}  d_{vu}.
\end{equation*}
\begin{Theorem}[\cite{APS}]
The map $D$ satisfies $D^2=0$, and the homology of $(\CKh(L),D)$ as a $\bZ\mathcal{C}$--graded group does not depend on the diagram $D$ or the order of the crossings.
\end{Theorem}

We will use $\SiKh(L)$ to denote the homology group of $(\CKh(L),D)$, and use $\SiKh(L;\bZ/2)$ to denote the homology group with $\bZ/2$--coefficients. The main results of the paper are the following theorems.

\begin{Theorem}
\label{thm_disjoint_from_annulus_intro}
	Suppose $L\subset (-1,1)\times T^2$ is a link, let $c$ be an essential simple closed curve on $T^2$. Then $L$ can be isotoped to a link disjoint from $(-1,1)\times c$ if and only if the gradings of $\SiKh(L;\bZ/2)$ are supported in $\bZ\{[c]\}\subset \bZ\mathcal{C}$.
\end{Theorem}

\begin{Theorem}
\label{thm_main_detection_result}
	Suppose $L$ is a non-empty link in $(-1,1)\times T^2$ such that 
	$$\rank_{\bZ/2} \SiKh(L;\bZ/2)\le 2.$$ Then $L$ is isotopic to 
a knot embedded in $\{0\}\times T^2$, and hence
$$\rank_{\bZ/2} \SiKh(L;\bZ/2)= 2.$$
\end{Theorem}

Our proofs of the theorems above rely heavily on the properties of the torus.

\begin{Question} 
Does the statement of Theorem \ref{thm_main_detection_result} still hold for the Khovanov skein homology of an arbitrary oriented compact surface?
\end{Question}

The proofs of Theorem \ref{thm_disjoint_from_annulus_intro} and Theorem \ref{thm_main_detection_result} follow the strategy of \cite{KM:Kh-unknot}: we establish a spectral sequence that relates $\SiKh(L;\bZ/2)$ to a singular instanton Floer homology group, then apply different excision properties of the instanton Floer homology group to prove the desired results. Since $[-1,1]\times T^2$ is not a priori a balanced sutured manifold in the sense of \cite{Juh:sut}, the non-vanishing theorem for sutured singular instanton Floer homology \cite[Theorem 7.12]{XZ:excision} cannot be applied in a straightforward way. This is discussed in Section \ref{sec_non_vanishing_closed_manifolds} where we prove a similar non-vanishing theorem for singular instanton Floer homology on closed manifolds. In Section \ref{sec_instanton_Floer_for_links}, we define  an intanton Floer homology invariant $\SiHI(L)$ for links  $L\subset (-1,1)\times \Sigma$, where $\Sigma$ is a closed oriented surface with arbitrary genus, and we study the basic properties of this invariant.  In Section \ref{sec_canonical_isomorphisms} and Section \ref{sec_spectral_sequence}, we show that when $\Sigma$ is a torus there exists a spectral sequence relating $\SiKh(L;\bZ/2)$ and $\SiHI(L)$. The sepectral sequence, together with the properties established in Section \ref{sec_instanton_Floer_for_links}, imply Theorem \ref{thm_disjoint_from_annulus_intro} and Theorem \ref{thm_main_detection_result}.
\\

{\bf Acknowledgement.} We would like to thank Maggie Miller for her help with the proof of Lemma \ref{lem_norm_minimizing_after_excision}.

\section{A non-vanishing theorem for closed manifolds}
\label{sec_non_vanishing_closed_manifolds}

This section establishes a non-vanishing theorem for singular instanton Floer homology on closed manifolds. The singular instanton Floer homology theory was developed by Kronheimer-Mrowka \cite{KM:Kh-unknot, KM:YAFT}, and the reader may refer to \cite[Section 2]{XZ:excision} for a brief review. Let $Y$ be an oriented closed 3-manifold, $L\subset Y$ be a link, $(\omega,\partial \omega)\subset (Y,L)$ be a properly embedded 1-manifold, and suppose $(Y,L,\omega)$ forms an admissible triple as defined in \cite[Section 2.1]{XZ:excision}, we use $\II(Y,L,\omega)$ to denote the singular instanton Floer homology group of the triple $(Y,L,\omega)$. All the instanton Floer homology groups in this paper are with $\bC$--coefficients unless otherwise specified. 

Let $\II(Y,L,\omega)_{(2)}$ be the intersection of the generalized eigenspaces of $\mu(\pt)$ with eigenvalue $2$ for all points in $Y$. 
Suppose $\Sigma\subset Y$ is an embedded oriented surface with genus $g$ that intersects $L$ transversely at $n$ points, we use $\II(Y,L,\omega|\Sigma)$ to denote the generalized eigenspace of $\muu(\Sigma)$ in $\II(Y,L,\omega)_{(2)}$ with eigenvalue $2g+n-2$.  

If $\Sigma$ is a compact surface possibly with boundary, we  use $x(\Sigma)$ to denote the Thurston norm \cite{Thurston1986anorm} of $\Sigma$. 
Namely, if $\Sigma$ is connected, then $x(\Sigma)$ is defined to be $\max\{0,-\chi(\Sigma)\}$, where $\chi(\Sigma)$ is the Euler characteristic of $\Sigma$; if $\Sigma$ is not connected, let $\Sigma_1,\cdots,\Sigma_k$ be the connected components of $\Sigma$, then $x(\Sigma)$ is defined to be $\sum_{i=1}^k x(\Sigma_i)$. 
Let $Y$ be an oriented compact 3-manifold possibly with boundary, and suppose $a\in H_2(Y,\partial Y;\bZ)$, then the \emph{Thurston norm of $a$} is defined to be the minimum value of $x(\Sigma)$ for all oriented, properly embedded surfaces $(\Sigma,\partial \Sigma)\subset (Y,\partial Y)$ whose fundamental classes are equal to $a$. 
We use $x(a)$ to denote the Thurston norm of $a$.

Let $L\subset Y$ be a link in the interior of $Y$, Scharleman \cite{Schar} introduced a generalized Thurston norm $x_L(\cdot)$ with respect to $L$ as follows. Suppose $(\Sigma, \partial \Sigma)\subset (Y,\partial Y)$ is a properly embedded surface that is transverse to $L$ and intersects $L$ at $n$ points, let $\chi_L(\Sigma)=n-\chi(\Sigma)$. In other words, $\chi_L(\Sigma)$ is the Euler characteristic of $\Sigma-L$. If $\Sigma$ is connected, let $x_L(\Sigma)=\max\{0,\chi_L(\Sigma)\}$; 
if $\Sigma$ is not connected, let $\Sigma_1,\cdots,\Sigma_k$ be the connected components of $\Sigma$, define $x_L(\Sigma)$ to be $\sum_{i=1}^k x_L(\Sigma_i)$. 
Let $a\in H_2(Y,\partial Y;\bZ)$, the \emph{generalized Thurston norm of $a$ with respect to $L$}, which is denoted by $x_L(a)$, is defined to be the minimum value of $x_L(\Sigma)$ for all oriented, properly embedded surfaces $(\Sigma,\partial \Sigma)\subset (Y,\partial Y)$ that are transverse to $L$ and have fundamental classes equal to $a$.

Suppose $(\Sigma,\partial \Sigma)\subset (Y,\partial Y)$ is a properly embedded oriented surface, 
and let $[\Sigma]\in H_2(Y,\partial Y;\bZ)$ be the fundamental class of $\Sigma$,
we say that $\Sigma$ \emph{minimizes the Thurston norm}, or $\Sigma$ is \emph{norm-minimizing}, if $x(\Sigma)=x([\Sigma])$. 
Similarly, if $L$ is a link in the interior of $Y$ that is transverse to $\Sigma$, we say that $\Sigma$ \emph{minimizes the generalized Thurston norm with respect to $L$}, or $\Sigma$ is \emph{$L$--norm-minimizing}, if $x_L(\Sigma)=x_L([\Sigma])$.

If $L$ is oriented, we use $\Sigma\cdot L$ to denote the algebraic intersection number of $\Sigma$ and $L$. Notice that it makes sense to refer to the parity of $\Sigma\cdot L$ without specifying the orientations of $\Sigma$ and $L$.

From now on, all the curves, surfaces, and 3-manifolds are assumed to be oriented unless otherwise specified, and all the maps are assumed to be smooth.

 The main result of this section is the following theorem.

\begin{Theorem}\label{thm_closed_nonvanish_singular}
Let $Y$ be a connected closed 3-manifold, let $L\subset Y$ be a link such that $Y-L$ is irreducible, and let $\omega\subset Y-L$ be an embedded closed 1-manifold.  Suppose $\Sigma\subset Y$ is a connected, oriented, closed, $L$--norm-minimizing surface in $Y$, such that $|\Sigma\cdot L|$ is odd and at least $3$. Then 
$$
\II(Y,L,\omega|\Sigma)\neq 0.
$$
\end{Theorem}

\begin{remark}
If $Y-L$ is reducible then the statement of Theorem \ref{thm_closed_nonvanish_singular} does not always hold. For example, let $L_0\subset Y_0$ be a link in a closed  3-manifold, and let $K\subset S^1\times S^2$ be the knot given by $S^1\times \{\pt\}$.
Let $Y=Y_0\# S^1\times S^2$, let $L\subset Y$ be given by the disjoint union of $L_0$ and $K$. Then the Chern-Simons functional has no critical point because there is no $\SU(2)$--representation of $\pi_1(Y-L-\omega)$ that satisfies the desired holonomy conditions, and hence $\II(Y,L,\omega)$ is zero for every possible $\omega$.
\end{remark}

Theorem \ref{thm_closed_nonvanish_singular} is analogous to the following result of Kronheimer and Mrowka.

\begin{Theorem}[{\cite[Theorem 7.21]{KM:suture}}]
\label{thm_closed_nonvanish_nonsingular}
	Let $Y$ be a connected closed irreducible 3-manifold, and let $\omega\subset Y$ be an embedded closed 1-manifold.  Suppose $\Sigma\subset Y$ is a connected, closed, norm-minimizing surface in $Y$ with genus at least $1$ such that $\Sigma\cdot \omega$ is odd, then 
	$$ \II(Y,\emptyset,\omega|\Sigma)\neq 0.$$
\end{Theorem}

\begin{remark}
If $Y$ is reducible then the statement of Theorem \ref{thm_closed_nonvanish_nonsingular} does not always hold. For example, let $Y_0$ be a connected closed oriented 3-manifold, let $Y=Y_0\# S^1\times S^2$, let $\omega\subset Y$ be given by $S^1\times \{\pt\}\subset S^1\times S^2$, then $\II(Y,\emptyset,\omega)=0$, because there is no $\SU(2)$ representation of $\pi_1(Y-\omega)$ satisfying the holonomy conditions and hence the  Chern-Simons functional has no critical point.
\end{remark}

Our strategy of proving Theorem \ref{thm_closed_nonvanish_singular} is to use the various excision formulas of singular instanton Floer homology to reduce the problem to Theorem \ref{thm_closed_nonvanish_nonsingular}. In order to apply the excision formulas, we need to show that the irreducibility of $Y-L$ and the $L$--norm-minimizing property of $\Sigma$ are preserved under the relevant excisions. This will be the established in Section \ref{subsec_excision_irreducibility} and Section \ref{subsec_excision_Thurston}. Most of the results in Sections \ref{subsec_excision_irreducibility} and \ref{subsec_excision_Thurston} follow from the standard techniques and some of them may be well-known to experts. Nevertheless, for the sake of completeness, we will present the proof when a direct reference is difficult to find.

\subsection{Excisions and irreducibility}
\label{subsec_excision_irreducibility}
Recall the following definition from \cite[Definition 1.2 (c)]{Schar}.
\begin{Definition}
Let $Y$ be a compact manifold possibly with boundary. Let $L\subset Y$ be a link in the interior of $Y$, and let $(\Sigma,\partial \Sigma)\subset (Y,\partial Y)$ be a properly embedded surface that is transverse to $L$. The surface $\Sigma$ is called \emph{$L$--compressble} if there exists an embedded closed disk $D\subset Y-L$, such that $\partial D=D\cap  (\Sigma-L)$, and $\partial D$ does not bound a disk on $\Sigma-L$. If $\Sigma$ is not $L$--compressble, then it is called \emph{$L$--incompressible}.
\end{Definition}

\begin{remark}
By Dehn's lemma, $\Sigma$ is $L$--incompressible if and only if the map $$\pi_1(\Sigma-L)\to\pi_1(Y-L)$$ induced by the inclusion is injective (see for example \cite[Corollary 3.3]{hatcher2000notes}).
\end{remark}

If $Y$ is a reducible 3-manifold, then an embedded sphere in $Y$ that does not bound a ball is called a \emph{reducing sphere} of $Y$. 

\begin{Lemma}
\label{lem_disjoint_reducing_sphere}
Suppose $Y$ is a closed  3-manifold, let $L\subset Y$ be a link, and let $\Sigma\subset Y$ be an $L$--incompressible surface. Suppose $Y-L$ is reducible, then there exists a reducing sphere of $Y-L$ that is disjoint from $\Sigma$.	
\end{Lemma}

\begin{proof}
Let $S$ be a reducing sphere of $Y-L$ that intersects $\Sigma$ transversely, and choose $S$ such that $S\cap \Sigma$ has the minimum number of components. Suppose $S\cap \Sigma\neq \emptyset$, then since $\Sigma$ is $L$--incompressible, every component of $S\cap \Sigma$ bounds a disk in $\Sigma-L$. Take an innermost component of $S\cap \Sigma$, we obtain a disk $D\subset \Sigma - L$ such that $D\cap S=\partial D$. The curve $\partial D$ divides $S$ into two closed disks $D_1$ and $D_2$, let $S_1=D\cup D_1$, $S_2=D\cup D_2$. Since $S$ is a reducing sphere, at least one of $S_1$ and $S_2$ is a reducing sphere. Perturbing $S_1$ and $S_2$ gives spheres that intersect $\Sigma$ at fewer number of components, which is contradictory to the definition of $S$. Therefore $S\cap \Sigma=\emptyset$.
\end{proof}

\begin{Lemma}
\label{lem_irreducible_after_gluing_along_incompressible_boundary}
Let $Y$ be an irreducible 3-manifold with boundary, and let $R_1$ and $R_2$ be two components of $\partial Y$ endowed with the boundary orientation. Suppose $R_1$ and $R_2$ are both incompressible. Let $\varphi:R_1\to R_2$ be an orientation-reversing diffeomorphism, and let $Y'$ be the the manifold obtained  from $Y$ by gluing $R_1$ and $R_2$ via $\varphi$, then $Y'$ is also irreducible.
\end{Lemma}

\begin{proof}
	Let $R'\subset Y'$ be the image of $R_1$ and $R_2$ in $Y'$, then $R'$ is incompressible. Suppose $Y'$ is reducible, applying Lemma \ref{lem_disjoint_reducing_sphere} with $L=\emptyset$, there exists a reducing sphere $S\subset Y'$ that is disjoint from $R'$. The pre-image of $S$ in $Y$ is then a reducing sphere of $Y$, contradicting the irreducibility of $Y$.
\end{proof}

Suppose $Y$ is a connected 3-manifold and $\Sigma \subset Y$ is a properly embedded connected surface. We say that $\Sigma$ is \emph{non-separating}, if $Y-\Sigma$ is connected. Otherwise, we say $\Sigma$ is \emph{separating}.

\begin{Lemma}
\label{lem_irreducible_after_excision_along_Sigma}
Suppose $Y$ is a connected closed 3-manifold and $L\subset Y$ is a link. Suppose $\Sigma$ is $L$--incompressible and non-separating in $Y$. Let $\varphi:\Sigma\to\Sigma$ be an orientation-preserving diffeomorphism of $\Sigma$ that maps $\Sigma\cap L$ to $\Sigma\cap L$, and let $Y'$ be the closed 3-manifold obtained by cutting $Y$ open along $\Sigma$ and gluing back via $\varphi$, let $L'$ be the image of $L$ in $Y'$. If $Y-L$ is irreducible, then $Y'-L'$ is also irreducible.
\end{Lemma}

\begin{proof}
Since $\Sigma$ is $L$--incompressible in $Y$, the surface $\Sigma'$ is $L'$--incompressible in $Y'$. Suppose $Y'-L'$ is reducible, then by Lemma \ref{lem_disjoint_reducing_sphere}, there is a reducing sphere $S'$ of $Y'-L'$ that is disjoint from $\Sigma'$. By the irreducibility of $Y-L$, the pre-image of $S'$ in $Y$ bounds a 3-ball $B$ in $Y-L$. Since $S'$ does not bound a ball in $Y'-L'$, we have $\Sigma\subset B$, which contradicts the the assumption that $\Sigma$ is non-separating.
\end{proof}

\begin{Lemma}
\label{lem_incompressibility_of_boundary_N(L)}
Suppose $Y$ is a connected closed 3-manifold, and $L\subset Y$ is a non-empty link such that $Y-L$ is irreducible.
 Let $N(L)$ the open tubular neighborhood of $L$. Furthermore, if $L$ has only one component, we assume that $Y-N(L)$ is not diffeomorphic to $S^1\times D^2$. Then the boundary of the tubular neighborhood of $\partial N(L)$ is incompressible in $Y-L$.
\end{Lemma}
\begin{proof}
	If $K$ is a component of $L$, let $N(K)$ be the corresponding component of $N(L)$. Suppose $\partial N(L)$ is compressible in $Y-L$, then there is a component $K$ of $L$ and a closed disk $D\subset Y-L-N(K)$, such that $D\cap \partial N(K)=\partial D$, and $\partial D$ is essential on $\partial N(K)$. Let $R$ be a regular neighborhood of $N(K)\cup D$, then $\partial R\cong S^2$. By the irreducibility of $Y-L$, $\partial R$ bounds a ball in $Y-L$. Since $Y$ is assumed to be connected, we conclude that $K$ is the only component of $L$, and $Y-N(K)$ is diffeomorphic to $S^1\times D^2$, which contradict the assumptions.
\end{proof}

\begin{Lemma}
\label{lem_irreducible_after_adding_parallel_copy}
Suppose $Y$ is a connected closed 3-manifold, and let $L\subset Y$ be a link such that $Y-L$ is irreducible.
 Suppose $L$ has at least two components. Let $K$ be a component of $L$, let $L'\subset Y$ be obtained from $L$ by adding a parallel copy of $K$ using an arbitrary framing of $L$. If $Y-L$ is irreducible, then $Y-L'$ is also irreducible.
\end{Lemma}
\begin{proof}
	Let $N(K)$ be a tubular neighborhood of $K$, then $Y-L'$ is obtained from $Y-L$ by removing $N(K)-K$ and gluing back a copy of $S^1\times (D^2-\{p_1,p_2\})$. By the assumptions, $Y-L-N(K)$ is irreducible. By Lemma \ref{lem_incompressibility_of_boundary_N(L)}, $\partial N(K)$ is incompressible in $Y-L-N(K)$. Notice that $S^1\times (D^2-\{p_1,p_2\})$ is an irreducible manifold with incompressible boundary. 
	 Therefore by Lemma \ref{lem_irreducible_after_gluing_along_incompressible_boundary}, $Y-L'$ is also irreducible.
\end{proof}

\begin{Lemma}
\label{lem_irreducibility_Y-L-m}
	Suppose $Y$ is a connected closed 3-manifold, and $L\subset Y$ is a non-empty link. Moreover, if $Y\cong S^1\times S^2$, we assume that $L$ is not isotopic to the knot given by $S^1\times \{\pt\}$. Let $K$ be a component of $L$, let $m\subset Y$ be a meridian of $K$. If $Y-L$ is irreducible, then $Y-L-m$ is also irreducible.
\end{Lemma}
\begin{proof}
	Let $N(K)\subset Y$ be an open tubular neighborhood of $K$ such that $\overline{N(K)}\cap L=K$. We discuss two cases.
	
	If $\partial N(K)$ is incompressible in $Y-L$, we isotope $m$ such that $m\subset N(K)$, then $N(K)-K-m$ is irreducible with incompressible boundary. By the assumptions, $Y-L-N(K)$ is irreducible. Therefore by Lemma \ref{lem_irreducible_after_gluing_along_incompressible_boundary}, 
	$$Y-L-m= \big(N(K)-K-m\big) \cup_{\partial N(K)} \big(Y-L-N(K)\big)$$
	is irreducible.
	
	If $\partial N(K)$ is compressible in $Y-L$, then by Lemma \ref{lem_incompressibility_of_boundary_N(L)}, $L$ has one component, and $Y-N(K)$ is diffeomorphic to $S^1\times D^2$. Since $S^1\times D^2$ is irreducible, $Y-L-m$ is irreducible as long as $m$ is not contractible in $Y-L$. If $m$ is contractible in $Y-L$, then $Y\cong S^1\times S^2$ and $L$ is  isotopic to the knot given by $S^1\times \{\pt\}$, which contradicts the assumptions.
\end{proof}

\subsection{Excisions and the generalized Thurston norm}
\label{subsec_excision_Thurston}

The following lemma was explained to the authors by Maggie Miller.
\begin{Lemma}[\cite{Miller_private_communication}]
\label{lem_norm_minimizing_after_excision}
Let $Y$ be a closed 3-manifold, let $\Sigma\subset Y$ be a closed, connected, norm-minimizing surface, and let $\varphi:\Sigma\to\Sigma$ be an orientation-preserving diffeomorphism of $\Sigma$. Suppose $Y'$ is the closed 3-manifold obtained by cutting $Y$ open along $\Sigma$ and gluing back via $\varphi$, and suppose $\Sigma'$ is the image of $\Sigma$ in $Y'$, then $\Sigma'$ is norm-minimizing in $Y'$.
\end{Lemma}

\begin{proof}
Without loss of generality, we may assume that $Y$ is connected.

If $x(\Sigma)=x(\Sigma')=0$, then $\Sigma'$ is obviously norm-minimizing. In particular, if $Y\cong S^1\times S^2$, then $x(\Sigma)=0$ and $\Sigma'$ is norm-minimizing.  In the following, we assume that $x(\Sigma)>0$, therefore $\Sigma$ being norm-minimizing implies that it is incompressible. 

If $Y$ is irreducible, then by \cite[Theorem 5.5]{G:Sut-1}, there exists a taut foliation on $Y$ such that $\Sigma$ is a closed leaf. Therefore, there is a taut foliation on $Y'$ such that $\Sigma'$ is a closed leaf. By \cite[Corollary 2]{Thurston1986anorm}, $\Sigma'$ is  norm-minimizing in $Y'$.

In general, we apply induction on the number of components in the prime decomposition of $Y$. If $Y$ is prime, the result follows from the previous arguments. If $Y$ is not prime, then by Lemma \ref{lem_disjoint_reducing_sphere}, there is a reducing sphere $S\subset Y$ such that $S\cap \Sigma=\emptyset$.

If $S$ is separating in $Y$, then $S$ decomposes $Y$ into $Y=Y_1\# Y_2$, and $\Sigma$ is contained in one of the components. Assume $\Sigma$ is contained in the component given by $Y_1$, let $Y_1'$ be the manifold obtained by cutting $Y_1$ open along $\Sigma$ and gluing back via the map $\varphi$, then $Y'\cong Y_1'\# Y_2$. By the induction hypothesis, $\Sigma'$ is norm-minimizing in $Y_1'$. Therefore it is also norm-minimizing in $Y_1'\# Y_2$ by the additivity of Thurston norm under connected sums.

If $S$ is non-separating in $Y$, recall that $\Sigma$ is norm-minimizing and we are assuming $x(\Sigma)>0$, therefore $\Sigma$ is not homologous to $S$. As a consequence, there is an embedded circle in $Y$ that intersects $S$ transversely at one point and is disjoint from $\Sigma$. In this case, $Y$ can be decomposed as $Y=Y_1\,\#\, S^1\times S^2$ such that $\Sigma\subset Y_1$, and the desired result follows from the induction hypothesis on $Y_1$.
\end{proof}

\begin{Lemma}
\label{lem_minimizer_intersect_torus}
Let $Y$ be a compact 3-manifold possibly with boundary, let $L\subset Y$ be a link in the interior of $Y$. Let $T=T_1\cup\cdots\cup T_k\subset Y-L$ be a disjoint union of tori, such that for each $i$, the torus $T_i$ is either an incompressible embedded torus in the interior of $Y-L$ or a component of $\partial Y$. Then given $a\in H_2(Y,\partial Y;\mathbb{Z})$,  there exists a properly embedded surface $(S,\partial S)\subset (Y,\partial Y)$ whose fundamental class equals $a$, such that $S$ is $L$--norm-minimizing,  $S$ and $T$ intersect transversely, and for each $i$, $S\cap T_i$ (possibly being empty) consists of parallel and coherently oriented essential curves on $T_i$, where $S\cap T_i$ is oriented by the orientations of $Y$, $S$ and $T_i$. 
\end{Lemma}
\begin{remark}
The surface $S$ given by Lemma \ref{lem_minimizer_intersect_torus} is not necessarily connected.
\end{remark}

\begin{proof}
Let $(S,\partial S)\subset (Y,\partial Y)$ be an $L$--norm-minimizing surface with fundamental class equal to $a$ that intersects $T$ transversely. Choose $S$ such that $S\cap T$ has the minimum number of connected components. Given $i$, if $S\cap T_i$ contains a trivial circle in $T_i$, let $c\subset S\cap T_i$ be an innermost circle, then $c$ bounds a closed disk $D$ on $T_i$ such that $D\cap S=c$. Compressing $S$ along $D$ yields a surface $S'$ such that $S'$ is homologous to $S$, $x_L(S')\le x_L(S)$, and $S'$ intersects $T$ transversely at fewer components, which is contradictory to the definition of $S$. 
Therefore $S\cap T_i$ consists of parallel essential curves on $T_i$. 

If the curves in $S\cap T_i$ are not all coherently oriented, then there exists a pair of  circles in $S\cap T_i$ that are oriented reversely and bound a closed annulus $A\subset T_i$ such that $A\cap S=\partial A$. 
Compressing $S$ along $A$ yields a surface that is homologous to $S$ and intersects $T$ at fewer components. Since $T_i$ is either incompressible in $Y-L$ or is a boundary component of $Y$, the resulting surface does not increase the generalized Thurston norm with respect to $L$, which is contradictory to the definition of $S$. In conclusion, $S\cap T_i$  consists of parallel and coherently oriented essential curves on $T_i$.
\end{proof}

\begin{Lemma}
\label{lem_norm_minimzing_after_gluing_torus_boundary}
Let $Y$ be a compact 3-manifold with boundary, and let $L$ be a link in the interior of $Y$. Suppose $T_1,T_2$ are two distinct torus components of $\partial Y$ that are both incompressible in $Y-L$. Suppose $(S,\partial S)\subset (Y,\partial Y)$ is $L$--norm-minimizing such that for $i=1,2$, $\partial S\cap T_i$  consists of a (possibly empty) family of parallel, coherently oriented, essential circles, and suppose $\partial S\cap T_1$ and $\partial S\cap T_2$ have the same number of components. 
Let $\varphi:T_1\to T_2$ be an orientation-reversing diffeomorphism such that $\varphi(S\cap T_1)= S\cap T_2$, and that $\varphi|_{S\cap T_1}$ is orientation-preserving with respect to the boundary orientation of $\partial S$. Let $Y'=Y/\varphi$, let $S'$ be the image of $S$ in $Y'$, and let $L'$ be the image of $L$ in $Y'$. Then $S'$ is $L'$--norm-minimizing.
\end{Lemma}
\begin{proof}
	Let $T$ be the image of $T_1$ and $T_2$ in $Y'$, then $T$ is incompressible in $Y'$. 
	Assume $S'$ is not $L'$--norm-minimizing, then by Lemma \ref{lem_minimizer_intersect_torus}, there is a surface $\widetilde{S}$ in $Y'$ that is homologous to $S'$, such that $x_{L'}(\widetilde{S})<x_{L'}(S')$, and $S'$ intersects $T$ transversely at a family of parallel, coherently oriented, essential circles. Therefore, after an isotopy of $\widetilde{S}$, we may assume $\widetilde{S}\cap  T=S'\cap T$. Notice that the fundamental class of $S$ in $H_2(Y,\partial Y;\bZ)$ is given by the image of $[S']\in H_2(Y',\partial Y';\bZ)$ via the map 
	$$H_2(Y',\partial Y')\to H_2(Y',T\cup \partial Y') \cong H_2(Y,\partial Y) .$$
	Therefore the pre-image of $\widetilde{S}$ in $Y$ is homologous to $S$, and it has a smaller generalized Thurston norm with respect to $L$, which contradicts the $L$--norm-minimizing property of $S$.
\end{proof}

\begin{Lemma}
\label{lem_norm_minimizing_after_gluing_along_tori}
Let $Y$ be a connected 3-manifold, and let $L\subset Y$ be a link in the interior of $Y$. Let $T$ be a separating torus in the interior of $Y$ that is disjoint from $L$. Let $Y_1$ and $Y_2$ be the closures of the two components of $Y-T$, and let $L_1=L\cap Y_1$, $L_2=L\cap Y_2$. Let $a\in H_2(Y,\partial Y;\mathbb{Z})$, and let $a_1,a_2$ be the images of $a$ under the restriction maps 
$$H_2(Y,\partial Y;\bZ)\to H_2(Y,\partial Y\cup Y_2;\bZ)\cong H_2(Y_1,\partial Y_1;\bZ)$$
 and 
 $$H_2(Y,\partial Y;\bZ)\to H_2(Y,\partial Y\cup Y_1;\bZ)\cong H_2(Y_2,\partial Y_2;\bZ)$$ respectively. Then
 $$x_L(a)\le x_{L_1}(a_1)+x_{L_2}(a_2).$$
 Moreover, if we further assume that $T$ is incompressible in $Y-L$, then 
  $$x_L(a)= x_{L_1}(a_1)+x_{L_2}(a_2).$$
\end{Lemma}
\begin{proof}
Let $(S_1,\partial S_1)\subset (Y_1,\partial Y_1)$, and $(S_2,\partial S_2)\subset (Y_2,\partial Y_2)$ be surfaces whose fundamental classes are $a_1$ and $a_2$ that realize the generalized Thurston norm with respect to $L_1$ and $L_2$ respectively. By Lemma \ref{lem_minimizer_intersect_torus}, we may take $S_1$, $S_2$ such that $T\cap \partial S_i$ consists of coherently oriented parallel essential simple closed curves on $T$. By the Mayer-Vietoris sequence, the fundamental classes of $\partial S_1$ and $\partial S_2$ are negative to each other in $H_1(T,\bZ)$, and hence one can isotope $S_1$ and $S_2$ such that $\partial S_1=\partial S_2$ with opposite orientations. Gluing $S_1$ and $S_2$ gives a properly embedded surface $S$ in $Y$. Since $H_2(T,\bZ)\cong \bZ$, by the Mayer-Vietoris sequence again, the kernel of the map 
$$
H_2(Y,\partial Y;\bZ)\to H_2(Y_1,\partial Y_1;\bZ) \oplus H_2(Y_2,\partial Y_2;\bZ) 
$$
is generated by the fundamental class of $T$, and hence there is an integer $m$ such that $[S]=a+m[T]$ in $H_2(Y,\partial Y;\bZ)$. Since $x_L([T])=0$, we have $x_L(a)=x_L([S])$, therefore $x_L(a)\le x_L(S) \le x_{L_1}(S_1)+x_{L_2}(S_2) = x_{L_1}(a_1)+x_{L_2}(a_2)$. 

If we further assume that $T$ is incompressible in $Y-L$, then $x_L(S)=x_{L_1}(S_1)+x_{L_2}(S_2)$, and by Lemma \ref{lem_norm_minimzing_after_gluing_torus_boundary}, $x_L(S)=x_L([S])=x_L(a)$. Therefore $x_L(a)=x_L(S)=x_{L_1}(S_1)+x_{L_2}(S_2)=x_{L_1}(a_1)+x_{L_2}(a_2)$.
\end{proof}

\begin{Lemma}
\label{lem_norm_minimizing_after_removing_link}
Let $Y$ be a compact 3-manifold, and let $L_1,L_2\subset Y$ be two disjoint links in the interior of $Y$.  Let $\Sigma\subset Y$ be an embedded closed surface that intersects $L_1\cup L_2$ transversely. Let $N(L_1)$ be an open tubular neighborhood of $L_1$ such that $L_2\cap N(L_1)= \emptyset$, and $\partial N(L)$ intersects $\Sigma$ transversely, and $N(L)\cap \Sigma$ is a disjoint union of meridian disks. Then $\Sigma$ is $(L_1\cup L_2)$--norm-minimizing in $Y$ if and only if $\Sigma-N(L_1)$ is $L_2$--norm-minimizing in $Y-N(L_1)$. 
\end{Lemma}

\begin{proof}
	To simplify the notation, write $L=L_1\cup L_2$.
	Suppose $\Sigma$ is not $L$--norm-minimizing, then there is a properly embedded surface $\Sigma'\subset Y$, such that $\Sigma'$ is homologous to $\Sigma$, and $\Sigma'$ intersects $L$ transversely, and $x_L(\Sigma')<x_L(\Sigma)$. After an isotopy of $\Sigma'$, we may assume that $\partial N(L_1)$ intersects $\Sigma'$ transversely and $N(L_1)\cap \Sigma'$ is a disjoint union of meridian disks. Therefore $x_{L_2}(\Sigma'-N(L_1))=x_L(\Sigma')<x_L(\Sigma)=x_{L_2}(\Sigma-N(L_1))$, and $\Sigma'-N(L_1)$ is homologous to $\Sigma-N(L_1)$ in $Y-N(L_1)$, and hence $\Sigma-N(L_1)$ is not $L_2$--norm-minimizing in $Y-N(L_1)$.

	Conversely, suppose $\Sigma-N(L_1)$ is not $L_2$--norm-minimizing in $Y-N(L_1)$, then there exists a surface $(S,\partial S)\subset (Y-N(L_1),\partial N(L_1))$ such that $(S,\partial S)$ is homologous to $(\Sigma-N(L_1),\partial N(L_1))$ in $(Y-N(L_1),\partial N(L_1))$, and 
	$$x_{L_2}(S)<x_{L_2}(\Sigma-N(L_1)).$$
	 By Lemma \ref{lem_minimizer_intersect_torus}, we may choose $S$ such that $\partial S$ is a disjoint union of coherently oriented meridian circles on $\partial N(L_1)$. Let $\Sigma'\subset Y$ be the closed surface obtained by filling $S$ with meridian disks in $N(L_1)$, then $\Sigma'$ is homologous to $\Sigma$ in $Y$, and 
	$$x_L(\Sigma')=x_{L_2}(S)<x_{L_2}(\Sigma-N(L_1))=x_L(\Sigma),$$
	and hence $\Sigma$ is not $L$--norm-minimizing in $Y$. Therefore the lemma is proved.
\end{proof}

\begin{Lemma}
\label{lem_norm_minimizing_after_adding_parallel_copy}
Let $Y$ be a connected closed 3-manifold, let $L\subset Y$ be a link such that $Y-L$ is irreducible, and suppose $L$ has at least two components. Let $\Sigma$ be a connected, $L$--norm-minimizing surface. Let $K$ be a component of $L$, let $K'$ be a parallel copy of $K$ with respect to an arbitrary framing of $K$, let $L'=L\cup K'$. If $\Sigma$ is $L$--norm-minimizing, then $\Sigma$ is also $L'$--norm-minimizing.
\end{Lemma}
\begin{proof}
		Let $N(K)$ be a small open tubular neighborhood of $K$,  then $(Y,L')$ is obtained from $(Y,L)$ by removing $(N(K),K)$ and gluing back a copy of $(S^1\times D^2,S^1\times\{p_1,p_2\})$. 
		Notice that by Lemma \ref{lem_norm_minimizing_after_removing_link} and \cite[Theorem 3]{Thurston1986anorm}, $\{\pt\}\times D^2\subset S^1\times D^2$ minimizes the generalized Thurston norm with respect to $S^1\times \{p_1,p_2\}$. By Lemma \ref{lem_norm_minimizing_after_removing_link}, $\Sigma-N(K)\subset Y-N(K)$ minimizes the generalized Thurston norm with respect to $L-K$. By Lemma \ref{lem_incompressibility_of_boundary_N(L)}, $\partial N(K)$ is incompressible in $Y-L-N(K)$. Hence by Lemma \ref{lem_norm_minimzing_after_gluing_torus_boundary}, $\Sigma$ is $L'$--norm-minimizing in $Y$.
\end{proof}

\begin{Definition}
Let $Y$ be an oriented 3-manifold, let $L\subset Y$ be a link in the interior of $Y$. Suppose $\Sigma\subset Y$ is an oriented embedded surface that intersects $L$ transversely. We say that $\Sigma$ intersects $L$ \emph{with the same sign}, if there exists an orientation of $L$ such that all the intersection points of $\Sigma$ and $L$ are positive.
\end{Definition}

\begin{Lemma}
\label{lem_norm_minimizing_singular_excision}
Let $Y$ be a closed oriented 3-manifold, and let $L\subset Y$ be a link. Suppose $\Sigma\subset Y$ is a connected, oriented, closed surface in $Y$ that intersects $L$ transversely with the same sign, and suppose that $\Sigma$ is $L$--norm-minimizing.	Let $\varphi:\Sigma\to\Sigma$ be an orientation-preserving diffeomorphism of $\Sigma$ that maps $\Sigma\cap L$ to $\Sigma\cap L$. Let $Y'$ be the closed 3-manifold obtained by cutting $Y$ open along $\Sigma$ and gluing back via $\varphi$, let $\Sigma'$ be the image of $\Sigma$ in $Y'$, and let $L'$ be the image of $L$ in $Y'$. Then $\Sigma'$ is $L'$--norm-minimizing.
\end{Lemma}

\begin{proof}
If $x_L(\Sigma)=0$, then $x_{L'}(\Sigma')=0$, and the result is obvious. From now on, we assume $x_L(\Sigma)>0$ and hence $H_2(Y;\bZ)\neq 0$ and $(Y,L)\neq (S^1\times S^2,S^1\times\{\pt\})$. Therefore by Lemma \ref{lem_incompressibility_of_boundary_N(L)}, $\partial N(L)$ is incompressible in $Y-N(L)$.

Let $N(L)$ be a small open tubular neighborhood of $L$. 
Construct a manifold $\overline{Y}$ as follows. Let $K_1,\cdots,K_m$ be the connected components of $L$. For each component $K_i$, let $N(K_i)$ be the corresponding component of $N(L)$. Let $F$ be a torus with an open disk removed. Remove $N(K_i)$ from $Y$ and glue back a copy of $S^1\times F$, such that $S^1\times\{\pt\}\subset S^1\times \partial F$ is glued to a longitude of $K_i$, and $\{\pt\}\times \partial F$ is glued to a meridian of $K_i$. Let $\overline{\Sigma}$ be the closed surface in $\overline{Y}$ obtained by replacing every meridian disk in $N(L)\cap \Sigma$ with a copy of $F$. 

Similarly, construct a manifold $\overline{Y'}$ from $(Y',L')$ and a closed surface $\overline{\Sigma'}\subset \overline{Y'}$ from $\Sigma$ as above.

Since $S^1\times F$ is foliated by parallel copies of $F$, $\{\pt\}\times F$ minimizes the Thurston norm in $S^1\times F$ \cite[Theorem 3]{Thurston1986anorm}. By Lemma \ref{lem_norm_minimizing_after_removing_link}, $\Sigma-N(L)$ minimizes the Thurston norm in $Y-N(L)$. Since $\Sigma$ intersects $L$ with the same sign and $\partial N(L)$ is incompressible in $Y-N(L)$, by Lemma \ref{lem_norm_minimzing_after_gluing_torus_boundary}, $\overline{\Sigma}$ minimizes the Thurston norm in $\overline{Y}$. By Lemma \ref{lem_norm_minimizing_after_excision}, $\overline{\Sigma'}$ minimizes the Thurston norm in $\overline{Y'}$, namely $x([\overline{\Sigma'}])=x(\overline{\Sigma'})$. 

Write $\overline{Y'}$ as $\overline{Y'_1}\cup_{\partial N(L')} \overline{Y'_2}$, where $\overline{Y'_1}=Y'-N(L')$, and $\overline{Y'_2}$ is the disjoint union of copies of $S^1\times F$.
By Lemma \ref{lem_norm_minimizing_after_gluing_along_tori}, we have 
$$x([\overline{\Sigma'}\cap \overline{Y'_1}])+x([\overline{\Sigma'}\cap \overline{Y'_2}])\ge  x([\overline{\Sigma'}]),$$ 
and by \cite[Theorem 3]{Thurston1986anorm}, we have $x([\overline{\Sigma'}\cap \overline{Y'_2}])=x(\overline{\Sigma'}\cap \overline{Y'_2})$. Therefore 
$$x([\Sigma'-N(L')])=x([\overline{\Sigma'}\cap \overline{Y'_1}])\ge x(\overline{\Sigma'}) -x(\overline{\Sigma'}\cap \overline{Y'_2}) = x(\overline{\Sigma'}\cap \overline{Y'_1})= x(\Sigma'-N(L')),$$ 
and hence  $\Sigma'-N(L')$ is norm-minimizing in $Y'-N(L')$. By Lemma \ref{lem_norm_minimizing_after_removing_link} again, this implies $\Sigma'$ is $L'$--norm-minimizing in $Y'$.
\end{proof}

\subsection{Proof of Theorem \ref{thm_closed_nonvanish_singular}}
We first prove a non-vanishing theorem when $L$ is disjoin from $\Sigma$.

\begin{Theorem}
\label{thm_closed_non_vanishing_L_disjoint_from_Sigma}
Let $Y$ be a connected closed 3-manifold, and let $\omega\subset Y$ be an embedded closed 1-manifold.  Let $L$ be a link disjoint from $\omega$ such that $Y-L$ is irreducible and $(Y,L)\ncong (S^1\times S^2,S^1\times \{\pt\})$. Suppose $\Sigma\subset Y-L$ is a connected closed surface with genus at least $1$, such that the algebraic intersection number $\Sigma\cdot\omega$ is odd, and that $\Sigma$ is $L$--norm-minimizing.  
	Then 
$$
\II(Y,L,\omega|\Sigma)\neq 0.
$$
\end{Theorem}

\begin{proof}
Let $\hat Y=S^1\times S^2$, take $\{p_1,p_2\}\subset S^2$, let $\hat L=S^1\times \{p_1,p_2\}\subset \hat Y$, let $\hat\omega$ be an arc in $\{\pt\}\times S^2\subset Y$ joining $(\pt,p_1)$ and $(\pt,p_2)$. Recall that $\II(\hat Y,\hat L,\hat \omega)_{(2)}$ denotes the generalized eigenspace of $\mu(\pt)$ on $\II(\hat Y,\hat L,\hat \omega)$ with eigenvalue $2$. By \cite[Proposition 2.1, Lemma 5.1]{XZ:excision}, we have
\begin{equation}
\label{eqn_instanton_rank_1_S2_times_S1}
	\II(\hat Y,\hat L,\hat \omega)\cong \II(\hat Y,\hat L,\hat \omega)_{(2)}\cong \bC.
\end{equation}

Let $K_1,\cdots,K_s$ be the connected components of $L$. For each $i=1,\cdots,s$, let $m_i$ be a small meridians around $K_i$, and let $u_i$ be an arc that connects $K_i$ and $m_i$. Let $m=\cup_{i=1}^s m_i$, let $u=\cup_{i=1}^s u_i$. Applying the un-oriented skein exact triangle \cite[Proposition 6.11]{KM:Kh-unknot} to a crossing between $K_1$ and $m_1$, we obtain
an exact triangle
\begin{equation*}
\xymatrix{
  \II(Y,L\cup m_1, \omega\cup u_1|\Sigma) \ar[r]^{} 
                &     \II(Y,L,\omega|\Sigma) \ar[d]^{}    \\
                & \II(Y,L,\omega|\Sigma)   \ar[ul]_{}             }
\end{equation*}
Therefore $\II(Y,L\cup m_1, \omega\cup u_1|\Sigma)\neq 0$ implies $\II(Y,L,\omega|\Sigma)\neq 0$. Repeating this argument on $m_2,\cdots,m_s$, we conclude that one only needs to prove
$$
\II(Y,L\cup m, \omega\cup u|\Sigma)\neq 0.
$$ 

For each $i=1,\cdots,s$, let $N(K_i)$ and $N(m_i)$ be small tubular neighborhoods of $K_i$ and $m_i$ respectively, such that $N(K_i)\cup N(m_i)$ is disjoint from $\omega$ and $u-u_i$, and $\partial N(K_i)$ and $\partial N(m_i)$ intersect $u_i$ transversely at one point. Moreover, suppose  $N(K_1),\cdots,N(K_s)$, $N(m_1),\cdots,N(m_s)$ are disjoint from each other and from $\Sigma$. Let $\varphi_i:\partial N(K_i)\to \partial N(m_i)$ be an orientation-reversing diffeomorphism that maps a meridian of $\partial N(K_i)$ to a meridian to $\partial N(m_i)$, and maps $\partial N(K_i)\cap u_i$ to $\partial N(m_i)\cap u_i$. 

Let $\overline{Y}$ be the manifold obtained from $Y$ by removing $N(K_i)$ and $N(m_i)$ for all $i$, and gluing the resulting boundaries using the maps $\varphi_i$. Let $\overline{\omega}$ and $\overline{\Sigma}$ be the images of $\omega$ and $\Sigma$ in $\overline{Y}$ respectively. By the torus excision theorem and \eqref{eqn_instanton_rank_1_S2_times_S1}, we have
$$
\II(Y,L\cup m,\omega\cup u|\Sigma)\cong \II(\overline{Y},\emptyset,\overline{\omega}|\overline{\Sigma})\otimes\Big(\II(\hat Y,\hat L,\hat \omega)_{(2)}\Big)^{\otimes s}\cong \II(\overline{Y},\emptyset,\overline{\omega}|\overline{\Sigma}). 
$$

Let $Y'=Y-\cup_{i=1}^s N(K_i)-\cup_{i=1}^s N(m_i)$.
By Lemma \ref{lem_irreducibility_Y-L-m}, $Y'$ is irreducible. By Lemma \ref{lem_incompressibility_of_boundary_N(L)}, both $\partial N(K_i)$ and $\partial N(m_i)$ are incompressible in $Y'$, and hence by Lemma \ref{lem_irreducible_after_gluing_along_incompressible_boundary}, $\overline{Y}$ is irreducible.
 Since $\Sigma$ is $L$--norm-minimizing in $Y$, it is norm-minimizing in $Y'$. By Lemma \ref{lem_norm_minimzing_after_gluing_torus_boundary}, $\overline{\Sigma}$ is norm-minimizing in $\overline{Y}$.
Notice that $\overline{\omega}\cdot  \overline{\Sigma}$ and $\omega\cdot \Sigma$ have the same parity, so $\overline{\omega}\cdot  \overline{\Sigma}$ is odd.
Therefore, it follows from Theorem \ref{thm_closed_nonvanish_nonsingular} that $\II(\overline{Y},\emptyset,\overline{\omega}|\overline{\Sigma})\neq 0$, and the theorem is proved.
\end{proof}

We also need the following technical lemma.

\begin{Lemma}
\label{lem_non_vanishing_even_number_of_intersections}
Let $Y$ be a connected closed 3-manifold, let $L\subset Y$ be a link, assume $Y-L$ is irreducible. Let $(\omega,\partial\omega)\subset (Y,L)$ be a properly embedded 1-manifold with boundary.  Suppose $\Sigma\subset Y$ is a connected closed surface in $Y$, such that $\Sigma$ is $L$--norm-minimizing and satisfies the following conditions:
\begin{enumerate}
\item The link $L$ can be decomposed as $L=L^{(1)}\sqcup L^{(2)}$, such that every component of $L^{(1)}$ intersects $\Sigma$ transversely at one point, and $L^{(2)}$ is disjoint from $\Sigma$,
\item $L^{(1)}$ is non-empty,
\item The interior of $\omega$ is disjoint from $L^{(1)}$. For each component $K$ of $L^{(1)}$, $\partial \omega$ and $K$ intersect at one point,
\item $L^{(2)}$ is disjoint from $\omega$,
\item $x_L(\Sigma)>0$.
\end{enumerate}
Then 
$$
\II(Y,L,\omega|\Sigma)\neq 0.
$$
\end{Lemma}

\begin{proof}
Let $K^{(1)}_1,\cdots,K^{(1)}_s$ be the components of $L^{(1)}$, and let $K^{(2)}_1,\cdots,K^{(2)}_t$ be the components of $L^{(2)}$. Conditions (2), (3), (4) above imply that $s$ is even and $s\ge 2$. Write $s=2s'$. Since each $K^{(1)}_i$ contains one boundary point of $\omega$, we can homotope $\omega$ such that $\omega$ and $\Sigma$ intersect transversely at an odd number of points. Let $N(K^{(1)}_i)$ be tubular neighborhoods of $K^{(1)}_i$, such that for each $i$, $N(K^{(1)}_i)$ intersects $\Sigma$ at a meridian disk, $\partial N(K^{(1)}_i)$ intersects $\omega$ and transversely at one point, and $N(K^{(1)}_i)$ is disjoint from $L^{(2)}$. 

For $i=1,\cdots,s'$, let $\varphi_i$ be an orientation-reversing diffeomorphism from $\partial N(K_i^{(1)})$ to $\partial N(K_{i+s'}^{(1)})$, such that 
$$\varphi_i\big(\partial N(K_i^{(1)})\cap \Sigma\big) = \partial N(K_{i+s'}^{(1)})\cap \Sigma,$$ and 
$$\varphi_i\big(\partial N(K_i^{(1)})\cap \omega\big)=\partial N(K_{i+s'}^{(1)})\cap \omega.$$ Let $\overline{Y}$ be the manifold obtained from $Y$ by removing $ N(K_i^{(1)})$ for all $i=1,\cdots,s$, and gluing the resulting boundaries by $\varphi_i$ for $i=1,\cdots, s'$. Let $\overline{L}$ be the image of $L^{(2)}$ in $\overline{Y}$, and let $\overline{\omega}$, $\overline{\Sigma}$ be the images of $\omega$, $\Sigma$ in $\overline{Y}$ respectively.

Let $\hat Y=S^1\times S^2$, take $\{p_1,p_2\}\subset S^2$, let $\hat L=S^1\times \{p_1,p_2\}\subset \hat Y$, let $\hat\omega$ be an arc in $\{\pt\}\times S^2\subset Y$ joining $(\pt,p_1)$ and $(\pt,p_2)$. By the torus excision theorem and \eqref{eqn_instanton_rank_1_S2_times_S1}, we have
$$
\II(Y,L,\omega|\Sigma)\cong \II(\overline{Y},\overline{L},\overline{\omega}|\overline{\Sigma})\otimes\Big(\II(\hat Y,\hat L,\hat \omega)_{(2)}\Big)^{\otimes s'}\cong \II(\overline{Y},\overline{L},\overline{\omega}|\overline{\Sigma}). 
$$

Since $Y-L$ is irreducible, by Lemma \ref{lem_incompressibility_of_boundary_N(L)}, $\partial N(K^{(1)}_i)$ and $\partial N(K^{(2)}_j)$ are incompressible in $Y-L$. Therefore, by Lemma \ref{lem_norm_minimzing_after_gluing_torus_boundary} and Lemma \ref{lem_norm_minimizing_after_removing_link}, $\overline{\Sigma}$ is $\overline{L}$--norm-minimizing in $\overline{Y}$. By Condition (4), $\overline{\omega}$ is a closed 1-manifold in $\overline{Y}-\overline{L}$. Since we have homotoped $\omega$ so that $\omega$ and $\Sigma$ intersect transversely at an odd number of points, $\overline{\omega}\cdot \overline{\Sigma}$ is odd. By Condition (5), $x_{\overline{L}}(\overline{\Sigma})=x(\overline{\Sigma}) = x_L(\Sigma)>0$. Since $\overline{\Sigma}$ is $\overline{L}$--norm-minimizing, $x_{\overline L}(\overline{\Sigma})>0$  implies $(\overline{Y},\overline{L})\neq (S^1\times S^2,S^1\times \{\pt\})$. By Theorem \ref{thm_closed_non_vanishing_L_disjoint_from_Sigma}, we have $\II(\overline{Y},\overline{L},\overline{\omega}|\overline{\Sigma})\neq 0$, and hence the lemma is proved.
\end{proof}

\begin{proof}[Proof of Theorem \ref{thm_closed_nonvanish_singular}]
 If $\Sigma$ and $L$ do not intersect with the same sign, we can take two consecutive points of $\Sigma\cap L$ on $L$ with opposite signs, and attach a tube to $\Sigma$ along the segment of $L$ bounded by the two intersection points. This  gives a connected surface $\Sigma'$, such that $[\Sigma']=[\Sigma]$, $x_L(\Sigma')=x_L(\Sigma)$, and $\Sigma'$ intersects $L$ at fewer points. Notice that $\II(Y,L,\omega|\Sigma)$ only depends on $[\Sigma]$ and $x_L(\Sigma)$, therefore $\II(Y,L,\omega|\Sigma)=\II(Y,L,\omega|\Sigma')$. By repeating this process, we may assume $\Sigma$ intersects $L$ with the same sign without loss of generality.

 Let $p_1,\cdots,p_n$ be intersection points of $\Sigma$ and $L$. By the assumptions, $n\ge 3$ and $n$ is odd. Since $n\ge 3$, we have $x_L(\Sigma)>0$, therefore $\Sigma$ being $L$--norm-minimizng implies that $\Sigma$ is non-separating in $Y$ and is $L$--incompressible. By Lemma \ref{lem_irreducible_after_excision_along_Sigma}, Lemma \ref{lem_norm_minimizing_singular_excision} and \cite[Theorem 6.6]{XZ:excision}, after taking an excision along $\Sigma$ if necessary, we may assume that $p_1,\cdots,p_n$ belong to distinct components of $L$ without loss of generality. Let $K_1,\cdots,K_n$ be the components of $L$ that contain $p_1,\cdots,p_n$ respectively.

Recall that $n\ge 3$ and $n$ is odd; let $n=2k+1$. Let $u_1,\cdots, u_k$ be disjoint arcs on $\Sigma-\{p_n\}$ such that $\partial u_i=\{p_{2i},p_{2i+1}\}$ for $i=1,\cdots,k$. Let $u=\cup_{i=1}^k u_i$. Excision along $\Sigma$ and \cite[Proposition 6.7]{XZ:excision} yields
	$$\II(Y,L,\omega|\Sigma)\cong \II(Y,L,\omega\cup u|\Sigma). $$
	
Let $N(K_1)$ be a tubular neighborhood of $K_1$ such that the intersection of $N(K_1)$ and $\Sigma$ is a meridian disk, $N(K_1)$ is disjoint from $L-K_1$ and $\omega$, and $\partial N(K_1)$ intersects $\omega\cup u$ transversely at one point. Fix a framing of $K_1$, let $K_1^{(1)}$, $K_1^{(2)}$ be two parallel copies of $K_1$ in $N(K_1)$ that are disjoint from $u$. 

Let $\hat Y=S^1\times S^2$, let $\{q_1,q_2,q_3,q_4\}\subset S^2$, and define $\hat{L}_2, \hat{L}_3, \hat{L}_4\subset \hat Y$ by $\hat{L}_2=S^1\times \{q_1,q_2\}$, $\hat{L}_3=S^1\times \{q_1,q_2,q_3\}$, $\hat{L}_4=S^1\times \{q_1,q_2,q_3,q_4\}$. 
Let $c_1\subset S^2-\{q_3,q_4\}$ be an arc joining $q_1$ and $q_2$, let $c_2\subset S^2$ be a circle that intersects $c_1$ transversely at one point and separates $\{q_1\}$ and $\{q_2,q_3,q_4\}$. Let $\hat{\omega}=\{\pt\}\times c_1\subset \hat Y$, let $\hat{T} = S^1\times c_2\subset \hat Y$, let $\hat \Sigma = \{\pt\}\times S^2\subset \hat Y$. 

Conducting excision on $(Y,L,\omega\cup u)\sqcup (\hat Y,\hat L_4,\hat\omega)$ along  $\partial N(K_1)\sqcup \hat {T}$, we have 
$$
	\II(Y,L,\omega\cup u|\Sigma)\otimes \II(\hat{Y},\hat{L}_4,\hat{\omega}|\hat\Sigma)
	\cong \II(Y,L\cup K_1^{(1)}\cup K_1^{(2)},\omega\cup u|\Sigma)\otimes \II(\hat{Y},\hat{L}_2,\hat{\omega}|\hat\Sigma).
$$
By \cite[Lemma 5.1]{XZ:excision}, we have 
\begin{equation}
\label{eqn_Ugn_rank_1_g=0_n=2,3,4}
	\II(\hat{Y},\hat{L}_2,\hat{\omega}|\hat\Sigma)\cong \II(\hat{Y},\hat{L}_3,\hat{\omega}|\hat\Sigma)\cong \II(\hat{Y},\hat{L}_4,\hat{\omega}|\hat\Sigma)\cong \bC,
\end{equation}
and hence
$$
\II(Y,L,\omega\cup u|\Sigma)\cong \II(Y,L\cup K_1^{(1)}\cup K_1^{(2)},\omega\cup u|\Sigma).
$$

Let $p_1^{(1)}=K_1^{(1)}\cap \Sigma$, $p_1^{(2)}=K_1^{(2)}\cap \Sigma$. Let $v$ be an arc on $\Sigma$ connecting $p_1^{(1)}$ and $p_n$ such that $v$ is disjoint from $\omega\cup u$. Excision along $\Sigma$ and \cite[Proposition 6.7]{XZ:excision} then yields

$$
\II(Y,L\cup K_1^{(1)}\cup K_1^{(2)},\omega\cup u|\Sigma)\cong \II(Y,L\cup K_1^{(1)}\cup K_1^{(2)},\omega\cup u\cup v|\Sigma).
$$

Let $N'(K_1)\subset N(K_1)$ be a solid torus that contains $K_1$ and $K_1^{(2)}$, such that $N'(K_1)$ is disjoint from $K_1^{(1)}\cup \omega\cup v$, and that $\partial N'(K_1)$ intersects $u$ transversely at one point. Conducting excision on $(Y,L\cup K_1^{(1)}\cup K_1^{(2)},\omega\cup u\cup v)\sqcup (\hat Y,\hat L_2,\hat\omega)$ along $\partial N'(K_1)\sqcup \hat T$ yields
$$
\II(Y,L\cup K_1^{(1)}\cup K_1^{(2)},\omega\cup u\cup v|\Sigma)\otimes \II(\hat Y,\hat L_2,\hat\omega)
\cong \II(Y,L\cup K_1^{(1)},\omega\cup u\cup v|\Sigma)\otimes \II(\hat Y,\hat L_3,\hat\omega).
$$
Therefore by \eqref{eqn_Ugn_rank_1_g=0_n=2,3,4},
$$
\II(Y,L\cup K_1^{(1)}\cup K_1^{(2)},\omega\cup u\cup v|\Sigma)\cong \II(Y,L\cup K_1^{(1)},\omega\cup u\cup v|\Sigma).
$$
By Lemma \ref{lem_irreducible_after_adding_parallel_copy}, $Y-L- K_1^{(1)} $ is irreducible. By Lemma \ref{lem_norm_minimizing_after_adding_parallel_copy}, $\Sigma$ minimizes the generalized Thurston norm with respect to $L\cup K_1^{(1)}$. Since $n\ge 3$, we have $x_{L\cup K_1^{(1)}}(\Sigma)>0$.
Therefore by Lemma \ref{lem_non_vanishing_even_number_of_intersections}, $\II(Y,L\cup K_1^{(1)},\omega\cup u\cup v|\Sigma)\neq 0$, and the theorem is proved.
\end{proof}

\section{Instanton Floer homology for links in thickened surfaces}
\label{sec_instanton_Floer_for_links}
In this section, let $\Sigma$ be a connected closed surface, and let $L$ be a link in $(-1,1)\times \Sigma$. The link $L$ defines a link in $([-1,1]/\{-1,1\})\times \Sigma \cong  S^1\times \Sigma$ by taking its image in the quotient space. When there is no risk of confusion, we will abuse notation and use $L$ to denote the image of $L$ in $S^1\times \Sigma$, and use $\{1\}\times\Sigma$ or $\Sigma$ to denote the image of $\{1\}\times \Sigma$ in $S^1\times \Sigma$.

Recall that all curves, surfaces, and 3-manifolds are assumed to be oriented unless otherwise specified, and all maps are assumed to be smooth.

\begin{Definition}
\label{def_SiHI}
Let $L$ be a link in $(-1,1)\times \Sigma$. Define 
$$
\SiHI(L):=\II(S^1\times \Sigma,L,S^1\times \{p\}|\Sigma),
$$
where $p$ is a point on $\Sigma$ such that $(S^1\times\{p\})\cap L=\emptyset$.
\end{Definition}
\begin{remark}
The fundamental class of $S^1\times \{p\}$ in $H_1(S^1\times \Sigma , L;\bZ/2)$ does not depend on $p$, therefore $\SiHI(L)$ is independent of the choice of $p$.
\end{remark}

\begin{remark}
\label{emk_SiHI_of_empty_link}
	By \cite[Proposition 7.4]{KM:suture}, we have $\SiHI(\emptyset)\cong \mathbb{C}$.
\end{remark}

\begin{Definition}
Given a (not necessarily oriented) link cobordism $S:L_1\to L_2$ in $[-1,1]\times (-1,1)\times \Sigma$,
define $\SiHI(S)$ to be the induced map on instanton Floer homology
$$
\SiHI(S)=\II([-1,1]\times S^1\times \Sigma, S, [-1,1]\times S^1\times \{\pt\}):
\SiHI(L_1)\to \SiHI(L_2).
$$
\end{Definition}

\begin{remark}
\label{rmk_sign_link_cobordism}
	Since the cobordism maps are only defined up to signs, 
	 $\SiHI(S)$ is a priori only well-defined up to an overall sign. However, if $S$ is an oriented cobordism between oriented links, there is a canonical choice of homology orientation that fixes the sign of $\SiHI(S)$ such that it satisfies the usual functoriality property. The reader may refer to \cite[Section 4.4]{KM:Kh-unknot} for more details.
\end{remark}

\begin{Lemma}
\label{lem_add_w2_on_Sigma}
Let $\Sigma$, $L$, $p$ be as in Definition \ref{def_SiHI}, let $c\subset \Sigma$ be an embedded closed 1-manifold such that $p\notin c$, let $\omega$ be the image of $[-1,1]\times \{p\}\cup \{1\}\times c$ in $S^1\times \Sigma$. Then 
$$
\SiHI(L) \cong \II(S^1\times \Sigma, L, \omega|\Sigma).
$$
\end{Lemma}

\begin{proof}
Let $\omega_1=S^1\times\{p\}\subset S^1\times\Sigma$, and let $\omega_2\subset S^1\times\Sigma$ be the image of $\{1\}\times c$, then $\omega=\omega_1\cup\omega_2$. By Definition \ref{def_SiHI}, $\SiHI(L)=\II(S^1\times\Sigma,L,\omega_1)$. Since the image of $\{1\}\times \Sigma$ is disjoint from $L$ and intersects $\omega_0$ transversely at one point, the excision formula \cite[Theorem 7.7]{KM:suture} yields
$$
\II(S^1\times \Sigma, L,\omega_1)\otimes \II(S^1\times\Sigma,\emptyset,\omega_1\cup\omega_2)\cong\II(S^1\times \Sigma, L, \omega|\Sigma).
$$
By \cite[Proposition 7.8]{KM:suture}, we have
$$
\II(S^1\times\Sigma,\emptyset,\omega_1\cup\omega_2)\cong \bC,
$$
and the lemma is proved.
\end{proof}
\begin{Definition}
	Suppose $L_1$ and $L_2$ are two links in $(-1,1)\times\Sigma$. Isotope $L_1$ such that $L_1\subset (-1,0)\times\Sigma$, and isotope $L_2$ such that $L_2\subset (0,1)\times \Sigma$. Define $L_1\sqcup L_2$ to be link given by the disjoint union of $L_1$ and $L_2$ after the isotopies.
\end{Definition}

\begin{remark}
	If the genus of $\Sigma$ is positive, then $L_1 \sqcup L_2$ may be non-isotopic to $L_2\sqcup L_1$. 
\end{remark}

\begin{Lemma}\label{lem_disjoint_union_of_links}
Suppose $L_1$ and $L_2$ are two links in $(-1,1)\times \Sigma$, then
	$$\SiHI(L_1\sqcup L_2)\cong \SiHI(L_1)\otimes \SiHI(L_2).$$
	Moreover, if $c\subset \Sigma$ is a simple closed curve, let $T_c=S^1\times c\subset S^1\times \Sigma$, then the isomorphism above is given by a map 
	$$
	\phi:\SiHI(L_1\sqcup L_2)\to \SiHI(L_1)\otimes \SiHI(L_2)
	$$
	such that $\phi\circ \muu(T_c)=\big(\muu(T_c)\otimes\id+\id\otimes\muu(T_c)\big)\circ\phi$.
\end{Lemma}
\begin{proof}
	Let $p\in \Sigma$ such that $S^1\times \{p\}$ is disjoint from $L_1$ and $L_2$.
	Excision on  $(S^1\times  \Sigma,L_1,S^1\times\{p\})\sqcup (S^1\times  \Sigma,L_2,S^1\times\{p\})$ along $\{1\}\times\Sigma\sqcup \{1\}\times\Sigma$ yields 
	$$
	\SiHI(L_1\sqcup L_2)\cong \SiHI(L_1)\otimes \SiHI(L_2).
	$$
	The isomorphism is induced by a cobordism map from $S^1\times\Sigma$ to $S^1\times \Sigma\sqcup S^1\times \Sigma$ where $T_c$ is cobordant to $T_c\sqcup T_c$, therefore the isomorphism intertwines the action of $\muu(T_c)$ on the left with $\muu(T_c)\otimes\id+\id\otimes\muu(T_c)$ on the right.
\end{proof}

Let $A\cong [-1,1]\times S^1$ be an annulus, recall that the annular instanton Floer homology was defined by \cite[Definition 4.1]{AHI} for a link $L\subset (-1,1)\times A$.

\begin{Lemma}
\label{lem_torus_reduce_to_AHI}
Suppose $\Sigma$ has genus $1$, let $c\subset \Sigma$ be an essential simple closed curve, let $N(c)\subset \Sigma$ be an open tubular neighborhood of $c$. Suppose $L\subset (-1,1)\times N(c)$. Let $L_{N(c)}$ be the link $L$ viewed as a link in $(-1,1)\times N(c)$, then 
\begin{equation}\label{eqn_torus_reduce_to_AHI}
	\SiHI(L)\cong \AHI(L_{N(c)}).
\end{equation}
\end{Lemma}

\begin{proof}
Let $p\in \Sigma-N(c)$. 
Let $c_1$ be one of the connected components of $\partial N(c)$,
let $c_2\subset \Sigma$ be a simple closed curve that intersects $c_1$ transversely at one point. 
Let $T=S^1\times c_1\subset S^1\times \Sigma$. 

Let $\hat{Y}=S^1\times S^2$. Write $S^2$ as the union of two copies of disks glued along the boundary, and label the two disks as $D^{(1)}$ and $D^{(2)}$. Let $q_1,q_2$ be two distinct points in the interior of $D^{(1)}$, let $\hat{L}_1\subset S^1\times D^{(1)}\subset \hat{Y}$ be the link given by $S^1\times \{q_1,q_2\}$.  Let $\hat{c}$ be a circle in $D^{(1)}$ separating $q_1$ and $q_2$, and let $\hat{T}=S^1\times \hat{c}\subset \hat{Y}$. Let $\hat{\omega}$ be an arc in $S^1\times D^{(1)}$ connecting the two components of $\hat{L}_1$ and intersecting $\hat{T}$ transversely at one point. Let $\hat{\Sigma}=\{\pt\}\times S^2\subset \hat{Y}$.

Notice that $N(c)$ is diffeomorphic to $(-1,1)\times S^1$. Choose a diffeomorphism from $(-1,1)\times (-1,1)$ to $D^{(2)}$, one obtains a diffeomorphism $\varphi$ from $(-1,1)\times N(c)$ to $D^{(2)}\times S^1 \cong S^1\times D^{(2)}$. Let $\hat{L}_2\subset S^1\times D^{(2)}$ be the link given by $\varphi(L)$. 

Recall that $\Sigma$ is a torus, and we identify $S^1$ with $[-1,1]/\{-1,1\}$.
Excision on $(S^1\times\Sigma, L, S^1\times\{p\}\cup\{1\}\times c_2)\sqcup (\hat{Y},\hat{L}_1,\hat{\omega})$ along $T\sqcup \hat{T}$ yields
$$
\II(S^1\times\Sigma,L,S^1\times\{p\}\cup \{1\}\times c_2|\Sigma)\otimes \II(\hat{Y},\hat{L}_1,\hat{\omega}|\hat{\Sigma})
\cong \II(\hat{Y},\hat{L}_1\cup \hat{L}_2,\hat{\omega}|\hat{\Sigma}).
$$
By Lemma \ref{lem_add_w2_on_Sigma}, $\II(S^1\times\Sigma,L,S^1\times\{p\}\cup \{1\}\times c_1|\Sigma)\cong \SiHI(L)$. By \cite[Proposition 5.1]{XZ:excision}, $\II(\hat{Y},\hat{L}_1,\hat{\omega}|\hat{\Sigma})\cong \bC$. By definition \cite[Definition 4.1]{AHI}, $\II(\hat{Y},\hat{L}_1\cup \hat{L}_2,\hat{\omega}|\hat{\Sigma})\cong \AHI(L_{N(c)})$, therefore the lemma is proved.
\end{proof}

Lemma \ref{lem_torus_reduce_to_AHI} also holds for surfaces with higher genera.
\begin{Lemma}
\label{lem_surface_reduce_to_AHI}
	Suppose $\Sigma$ has genus $g\ge 1$, let $c\subset \Sigma$ be a non-separating simple closed curve, let $N(c)\subset \Sigma$ be an open tubular neighborhood of $c$. Suppose $L\subset (-1,1)\times N(c)$. Let $L_{N(c)}$ be the link $L$ viewed as an annular link in $(-1,1)\times N(c)$. Then $\SiHI(L)\cong \AHI(L_{N(c)})$.
\end{Lemma}

\begin{proof}
	Let $c_1, c_2$ be the two components of $\partial N(c)$,  let $c_3$ be a simple closed curve on $\Sigma$ that intersects $c_1$ and $c_2$ transversely at one point. Let $p\in\Sigma-(N(c)\cup c_3)$.
	
	Let $T$ be an (abstract) torus, let $c'\subset T$ be an essential simple closed curve on $T$, let $p'\in T-c'$ be a point. Let $N(c')$ be an open tubular neighborhood of $c'$ in $T$, fix a diffeomorphism from $N(c)$ to $N(c')$, one obtains a diffeomorphism $\varphi$ from $(-1,1)\times N(c)$ to $(-1,1)\times N(c')$. Let $c_1'\subset T$ be a simple closed curve that intersects $c$ transversely at one point. Let $L'=\varphi(L)$.
	
	As before, we identify $S^1$ with $[-1,1]/\{-1,1\}$.
	 Excision on $(S^1\times \Sigma,L,S^1\times \{p\}\cup \{1\}\times c_3)$ along $S^1\times c_1\sqcup S^1\times c_2$ yields
	\begin{multline*}
			\II(S^1\times \Sigma, L, S^1\times \{p\} \cup \{1\}\times c_3) 
			\\
			\cong \II(S^1\times\Sigma, \emptyset, S^1\times \{p\} \cup \{1\}\times c_3) \otimes \II(S^1\times T,L',\{1\}\times c_1').	
	\end{multline*}
	
By Lemma \ref{lem_add_w2_on_Sigma}, $\II(S^1\times \Sigma, L, S^1\times \{p\} \cup \{1\}\times c_3)\cong \SiHI(L)$. 
By Remark \ref{emk_SiHI_of_empty_link} and Lemma \ref{lem_add_w2_on_Sigma}, $\II(S^1\times\Sigma, \emptyset, S^1\times \{p\} \cup \{1\}\times c_3)\cong \SiHI(\emptyset)\cong \bC$. Notice that the triple $(S^1\times T,L',\{1\}\times c_1')$ is diffeomorphic to the triple $(S^1\times T,L',S^1\times p')$, therefore by Lemma \ref{lem_torus_reduce_to_AHI}, $\II(S^1\times T,L',\{1\}\times c_1')\cong \AHI(L_{N(c)})$, and the lemma is proved.
\end{proof}

Lemma \ref{lem_surface_reduce_to_AHI} has the following immediate conseqences.
\begin{Lemma}
\label{lem_surface_knot_rank_2}
Let $c\subset \Sigma$ be a contractible or non-separating simple closed curve, and let $K=\{0\}\times c\subset S^1\times\Sigma$. Then $\SiHI(K)\cong \mathbb{C}^2$.
\end{Lemma}
\begin{proof}
	The lemma follows from Lemma \ref{lem_surface_reduce_to_AHI} and \cite[Example 4.2]{AHI}.
\end{proof}

\begin{Lemma}
\label{lem_rank_at_least_2^n_for_contractible_links}
If $L\subset (-1,1)\times\Sigma$ is a link contained in a 3-ball, then 
$$\dim \SiHI(L)\ge 2^n,$$ where $n$ is the number of components of $L$.
\end{Lemma}
\begin{proof}
	The result follows from Lemma \ref{lem_surface_reduce_to_AHI} and \cite[Lemma 8.1]{XZ:excision}
\end{proof}

The following lemma is a useful property of the $\muu$--map.

\begin{Lemma}
\label{lem_spectrum_symmetric_at_zero}
	Let $(Y,L,\omega)$ be admissible, and let $\Sigma\subset Y$ be a surface. Let $\lambda\in \bC$, then the dimension of the generalized eigenspace of $\muu(\Sigma)$ on $\II(Y,L,\omega)$ with eigenvalue $\lambda$ is the same as the dimension of the generalized eigenspace of $\muu(\Sigma)$ on $\II(Y,L,\omega)$ with eigenvalue $-\lambda$.
\end{Lemma}

\begin{proof}
	Recall that $\II(Y,L,\omega)$ is a relatively graded $\bZ/4$ linear space over $\bC$, and $\muu(T_c)$ is an operator with degree $2$. Decompose $\II(Y,L,\omega)$ as $I_1\oplus I_2\oplus I_3\oplus I_4$ with respect to the relative $\bZ/4$--gradings, let $\varphi$ be an automorphism of $\II(Y,L,\omega)$ such that $\varphi=\id$ on $I_1\oplus I_2$ and $\varphi=-\id$ on $I_3\oplus I_4$, then $\varphi$ gives an isomorhism from the generalized eigenspace of $\muu(\Sigma)$ with eigenvalue $\lambda$ to the generalized eigenspace of $\muu(\Sigma)$ with eigenvalue $-\lambda$.
\end{proof}

\begin{Lemma}
\label{lem_eigenspace_dimension_one_surface_knot}
	Suppose $\Sigma$ has genus $g\ge 1$, let $c_0\subset \Sigma$ be a non-separating simple closed curve, let $K=\{0\}\times c_0\subset S^1\times \Sigma$. Let $c$ be a simple closed curve on $\Sigma$ that intersects $c_0$ transversely at one point, and let $T_c=S^1\times c$. Then the eigenvalues of $\muu(T_c)$ on $\SiHI(K)$ are $\{-1,1\}$. Moreover, for each $\lambda\in\{-1,1\}$, the generalized eigenspace of $\muu(T_c)$ on $\SiHI(K)$ with eigenvalue $\lambda$ has dimension one.
\end{Lemma}

\begin{proof}
By \cite[Proposition 6.1]{XZ:excision}, the spectrum of $\muu(T_c)$ on $\SiHI(K)$ is a subset of $\{-1,1\}$. By Lemma \ref{lem_surface_knot_rank_2}, $\dim \SiHI(K)=2$. Therefore by Lemma \ref{lem_spectrum_symmetric_at_zero}, for each $\lambda\in\{-1,1\}$, the generalized eigenspace  with eigenvalue $\lambda$ has dimension one.
\end{proof}

\begin{Lemma}
\label{lem_range_of_eigenvalues}
Let $c$ be a simple closed curve on $\Sigma$, and let $T_c=S^1\times c$. 
Consider the operator $\muu(T_c)$ on $\SiHI(L)$. Suppose $T_c$ intersects $L$ transversely at $n$ points, then the spectrum of $\muu(T_c)$ on $\SiHI(L)$ is a subset of $\{-n+2i|i\in \bZ, 0\le i \le n\}$.
\end{Lemma}

\begin{proof}
	If $n$ is odd, then the result is a special case of \cite[Proposition 6.1]{XZ:excision}. If $n$ is even and $c$ is null-homologous, then $\muu(T_c)$ is zero and the statement is obvious. If $n$ is even and $c$ has a non-trivial fundamental class, then there exists a simple closed curve $c_1\subset \Sigma$ that intersects $c$ transversely at one point. Let $K_1=\{0\}\times c_1$. By the previous argument, the desired property holds for $L\sqcup K_1$. The result for $L$ then follows from Lemma \ref{lem_disjoint_union_of_links} and Lemma \ref{lem_eigenspace_dimension_one_surface_knot}.
\end{proof}

\begin{Definition}
Let $c$ be a simple closed curve on $\Sigma$, define the \emph{$c$--grading} on $\SiHI(L)$ to be the grading given by the generalized eigenspace decomposition of $\SiHI(L)$ with respect to $\muu(T_c)$.
\end{Definition}

\begin{remark}
By Lemma \ref{lem_range_of_eigenvalues}, if $T_c$ intersects $L$ transversely at $n$ points, then the $c$--grading of $\SiHI(L)$ is supported in $\{-n+2i|i\in \bZ, 0\le i \le n\}$.
\end{remark}

\begin{Proposition}
\label{prop_c_grading_zero_implies_disjoint}
	Suppose $\Sigma$ is a torus, let $c$ be a simple closed curve on $\Sigma$. Let $L$ be a link in $(-1,1)\times \Sigma$. Suppose the $c$--grading of $\SiHI(L)$ is supported at zero, then $L$ can be isotoped so that it is disjoint from $(-1,1)\times c$.
\end{Proposition}

\begin{proof}
	The statement is obvious if $c$ is null-homotopic. In the following, assume $c$ is an essential curve on $\Sigma$. The link $L$ can be decomposed as $L=L^{(1)}\sqcup L^{(2)}$, where $L^{(2)}$ is contained in a 3-ball in $(-1,1)\times \Sigma$, and $\big((-1,1)\times \Sigma\big)-L^{(1)}$ is irreducible. By Lemma \ref{lem_irreducible_after_gluing_along_incompressible_boundary}, $(S^1\times\Sigma )-L^{(1)}$ is also irreducible. 
	
	By Lemma \ref{lem_rank_at_least_2^n_for_contractible_links}, $\SiHI(L^{(2)})\neq 0$. By Lemma \ref{lem_range_of_eigenvalues}, the $c$--grading of $\SiHI(L^{(2)})$ is supported at degree zero, therefore by Lemma \ref{lem_disjoint_union_of_links} and the assumptions, $\SiHI(L^{(1)})$ is supported at $c$--degree zero. We only need to show that $L^{(1)}$ can be isotoped so that it is disjoint from $(-1,1)\times c$.
	
	Let $c_1$ be a simple closed curve on $\Sigma$ that intersects $c$ transversely at one point. For $n\in\bZ^+$, let $L_n\subset S^1\times \Sigma$ be the link contained in $\{1\}\times\Sigma$ given by $n$ parallel copies of $c_1$. Since every component of $L_n$ is non-contractible and $S^1\times \Sigma$ is irreducible, we have $(S^1\times\Sigma)-(L^{(1)}\sqcup L_n)$ is irreducible for all $n$.
	
	Now let $(S,\partial S)\subset ([-1,1]\times \Sigma,\{-1,1\}\times \Sigma)$ be a properly embedded surface such that $S$ is homologous to $[-1,1]\times c$ in $H_2([-1,1]\times \Sigma,\{-1,1\}\times \Sigma)$, $S$ is transverse to $L^{(1)}$, and $S$ minimizes the generalized Thurston norm with respect to $L^{(1)}$. By Lemma \ref{lem_minimizer_intersect_torus}, one can choose $S$ such that $\partial S=\{-1,1\}\times c$. Notice that the surface $S$ may be  
	non-connected. Since the fundamental classes of $\{-1\}\times c$ and $\{1\}\times c$ are both non-trivial in $H_1([-1,1]\times \Sigma)$, the two boundary components of $S$ must belong to the same component. Let $S_0$ be the connected component of $S$ such that $\partial S_0=\partial S$, then $S_0$ is also $L^{(1)}$-norm-minimizing. Since the map 
	$$H_2([-1,1]\times\Sigma)\to H_2([-1,1]\times \Sigma,\{-1,1\}\times \Sigma)$$
	is zero, every closed component of $S$ is null-homologous in $H_2([-1,1]\times \Sigma,\{-1,1\}\times \Sigma)$, therefore $S_0$ is homologous to $S$ in $H_2([-1,1]\times \Sigma,\{-1,1\}\times \Sigma)$.

	Let $\overline{S_0}$ be the image of $S_0$ in $S^1\times \Sigma$, let $T_c=S^1\times c$. Notice that the kernel of the map 
	$$
	H_2(S^1\times\Sigma) \to H_2(S^1\times\Sigma,\{1\}\times\Sigma)\cong H_2([-1,1]\times\Sigma,\{-1,1\}\times\Sigma)
	$$
	is generated by the fundamental class of $\{1\}\times\Sigma$.
	Therefore, there is an integer $m$ such that $[\overline{S_0}]=[T_c]+m\cdot [\{1\}\times \Sigma]$. 
	Recall that $\Sigma$ is a torus and $c_1$ is a simple closed curve on $\Sigma$ that intersects $c$ transversely at one point. By twisting $[-1,1]\times \Sigma$ near $\{1\}\times\Sigma$ in the direction of $c_1$, one can find a surface $S_0'$ such that $\partial S_0'=\partial S_0$, and $S_0'$ is isotopic to $S_0$ in $[-1,1]\times\Sigma$ (\emph{without} fixing the boundary), $x_{L^{(1)}}(S_0')=x_{L^{(1)}}(S_0)$, and the image of $S_0'$ in $S^1\times\Sigma$ is homologous to $T_c$. Let $\overline{S_0'}$ be the image of $S_0'$ in $S^1\times \Sigma$. By Lemma \ref{lem_norm_minimzing_after_gluing_torus_boundary}, $\overline{S_0'}$ is $L^{(1)}$-norm-minimizing in $S^1\times\Sigma$. 
	
	Suppose $\overline{S_0'}$ has genus $g$, then $g\ge 1$.
	Suppose $S_0'$ and $L^{(1)}$ intersect transversely at $n_1$ points. Let $n_2\ge 0$ be an integer such that $n_1+n_2$ is odd and at least 3. Since $\overline{S_0'}$ is $L^{(1)}$-norm-minimizing, it is also $(L^{(1)}\sqcup L_{n_2})$-norm-minimizing. Let $p\in \Sigma$ be a point such that $S^1\times p$ is disjoint from $L^{(1)}\sqcup L_{n_2}$. By Theorem \ref{thm_closed_nonvanish_singular}, 
	$$
	\II(S^1\times\Sigma,L^{(1)}\sqcup L_{n_2},S^1\times\{p\}|\overline{S_0'})\neq 0.
	$$
	By definition, this means the generalized eigenspace of $\muu(\overline{S_0'})$ on $\SiHI(K\sqcup L_{n_2})$ with eigenvalue $2g+n_1+n_2-2$ is non-trivial. 
	Since $\overline{S_0'}$ is homologous to $T_c$, we have $\muu(T_c)=\muu(\overline{S_0'})$, therefore by Lemma \ref{lem_disjoint_union_of_links} and Lemma \ref{lem_eigenspace_dimension_one_surface_knot}, we conclude that the generalized eigenspace of $\muu(T_c)$ on $\SiHI(L^{(1)})$ with eigenvalue $2g+n_1-2$ is non-trivial. Since the $c$--grading of $\SiHI(L^{(1)})$ is supported at degree zero, this implies $2g+n_1-2=0$, and hence $g=1$, $n_1=0$. In conclusion, $S_0'$ is an annulus such that $\partial S_0=\{-1,1\}\times c$, and $S_0'\cap L^{(1)}=\emptyset$.

	Notice that $[-1,1]\times \Sigma$ can be viewed as an $S^1$-fiber bundle with fibers parallel to $c$. By a standard property of Seifert-fibered spaces (see, for example, \cite[Proposition 1.11]{hatcher2000notes}), $S_0'$ can be isotoped to a surface that is either tangent to the fibers or transverse to the fibers. Since $\partial S_0'=\{-1,1\}\times c$, it cannot be isotoped to a surface that is transverse to the fibers, therefore  $S_0$ is isotopic to $[-1,1]\times c$, and the result is proved.
\end{proof}

\begin{Proposition}\label{prop_inst_rank_2_disjoint_from_c}
Suppose $\Sigma$ is a torus, $L \subset (-1,1)\times \Sigma$ is a link, and
$$
\dim_\bC\SiHI(L)\le 2.
$$
Then there is an essential simple closed curve $c$ on $\Sigma$, such that $L$ can be isotoped so that it is disjoint from $(-1,1)\times c$. 
\end{Proposition}
\begin{proof}
Let $c_1$, $c_2$ be a pair of essential simple closed curves on $\Sigma$ such that $c_1$ and $c_2$ intersect transversely at one point. If $\SiHI(L)$ is supported at degree zero with respect to either the $c_1$--grading or the $c_2$--grading, then the result follows from Proposition \ref{prop_c_grading_zero_implies_disjoint}. Otherwise, by Lemma \ref{lem_spectrum_symmetric_at_zero}, we must have $\dim_\bC\SiHI(L)= 2$. Moreover, since $\muu(S^1\times c_1)$ and $\muu(S^1\times c_2)$ are commutative, the simultaneous eigenvalues of $\big(\muu(S^1\times c_1),\muu(S^1\times c_2)\big)$  on $\SiHI(K)$ can be written as $(\lambda_1,\lambda_2)$ and $(-\lambda_1,-\lambda_2)$. By Lemma \ref{lem_range_of_eigenvalues}, there is a coprime pair of integers $(p,q)$ such that $p\lambda_1+q\lambda_2=0$. Let $c$ be a simple closed curve on $\Sigma$ with the homology class $p[c_1]+q[c_2]$, then the $\muu(S^1\times c)=p\muu(S^1\times c_1) + q\muu(S^1\times c_1)$ is identically zero on $\SiHI(L)$, and the result follows from Proposition \ref{prop_c_grading_zero_implies_disjoint}.
\end{proof}

\section{Canonical isomorphisms of instanton Floer homology}
\label{sec_canonical_isomorphisms}

This section constructs several maps that realize the isomorphisms given by Lemma \ref{lem_disjoint_union_of_links} and Lemma \ref{lem_torus_reduce_to_AHI}. These maps will play an important role in the computation of the spectral sequence in Section \ref{sec_spectral_sequence}. Recall that all curves, surfaces, and 3-manifolds are assumed to be oriented unless otherwise specified, and all maps are assumed to be smooth. Throughout this section, $\Sigma$ will denote a torus. We start with the following definition.

\begin{Definition}	\label{def_generator_u_0}
Let $p_1,p_2$ be two distinct points on $S^2$, let $(\hat Y,\hat L)$ be given by $(S^1\times S^2, S^1\times \{p_1,p_2\})$, and let $\hat \omega$ be an arc connecting the two components of $\hat L$. The critical set of the unperturbed Chern-Simons functional of $(\hat Y,\hat L,\hat \omega)$ consists of one regular point. Fixing a choice of orientation, let $\bfu_0$ be the generator of $\II(\hat Y,\hat L,\hat \omega)\cong \bC$ represented by the critical point.
\end{Definition}

\begin{remark}
\label{rmk_AHI_emptyset}
	The choice of $\bfu_0$ is unique up to a sign.
	By definition $\II(\hat Y,\hat L,\hat \omega)=\AHI(\emptyset)$, and the $\bfu_0$ defined above agrees with the choice of the generator of $\AHI(\emptyset)$ in \cite{AHI}, which was also denoted by $\bfu_0$.
\end{remark}

	Recall that $\Sigma$ is an (oriented) torus.
	Let $c\subset \Sigma$ be an (oriented) essential simple closed curve, let $c'\subset \Sigma$ be an (oriented) essential simple closed curve that intersects $c$ transversely at one point so that the algebraic intersection number $c\cdot c'=+1$. Let $N(c)\subset \Sigma$ be a tubular neighborhood of $c$. After an isotopy of $\Sigma$, we may assume that $c$, $\partial N(c)$, and $c'$ are all linearly embedded. 
	Let $L\subset (-1,1)\times \Sigma$ be a link embedded in $\{0\}\times N(c)$. Let $p\in \Sigma-N(c)$, then $\SiHI(L)=\II(S^1\times \Sigma, L,S^1\times \{p\}|\Sigma)$. 
		
	Let $T\subset S^1\times \Sigma$ be a linearly embedded torus such that its fundamental class is $[S^1\times c]+[\Sigma]$ and that it is disjoint from $\{0\}\times N(c)$. 	 Let $\pi:S^1\times\Sigma\to \Sigma$ be the projection map. 
	
	Let $L_{N(c)}$ be the link $L$ viewed as an annular link in $(-1,1)\times N(c)\cong (-1,1)\times A$, where the diffeomorphism from $N(c)$ to $A$ takes $N(c)\cap c'$ to $[-1,1]\times\{\pt\}$.

	Let $\hat Y,\hat L,\hat \omega,p_1,p_2$ be as in Definition \ref{def_generator_u_0}. Let $\hat c\subset S^2$ be an (oriented) circle separating $p_1$ and $p_2$ such that $S^1\times \hat c$ intersects $\hat \omega$ transversely at one point. Let $\hat T=S^1\times \hat c$.

	 Take the excision on $(\hat Y,\hat L,\hat \omega)\sqcup (S^1\times \Sigma, L,S^1\times\{\pt\})$ along $\hat T\sqcup T$ using a diffeomorphism from $T$ to $\hat T$, which maps $T\cap (\{0\}\times\Sigma)$ to $S^1\times \{\pt\}\subset \hat T=S^1\times \hat c$,  maps $T\cap \pi^{-1}(c')$ to $\{\pt\}\times \hat c\subset \hat T=S^1\times \hat c$, and preserves the orientations of $c'$ and $\hat c$. 
	 	 Define $\Phi_{c,c'}$ to be the composition map
	\begin{align}
		\II(S^1\times \Sigma, L,S^1\times\{\pt\}|\Sigma)
		& \stackrel{\otimes \bfu_0}{\longrightarrow} 
		\II(S^1\times \Sigma, L, S^1\times\{\pt\}|\Sigma)\otimes \II(\hat Y,\hat L,\hat \omega) 
		\nonumber
		\\
		& \longrightarrow \AHI(L_{N(c)}) \label{eqn_Phi_c,c'_essential},	
	\end{align}
	where the second map is induced by the excision cobordism above and the image of $\{0\}\times N(c)$ is identified with the annulus $\{0\}\times A\subset(-1,1)\times A$.

	Since the cobordism maps are defined up to a sign, the map $\Phi_{c,c'}$ is well-defined up to a multiplication by $\pm 1$. 
	By the torus excision theorem, $\Phi_{c,c'}$ is an isomorphism. It is straightforward to verify that $\Phi_{c,c'}$ is determined by the isotopy classes of $c, c'$, and the embedding of $L$ in $\{0\}\times N(c)$. 
	
\begin{Definition}
\label{def_tau_orientation}
	Suppose $\Sigma$ is a torus, recall that $\mathcal{C}$ denotes the set of isotopy classes of un-oriented essential simple closed curves on $\Sigma$. Define $\tau$ and $\mathfrak{o}$ as follows.
	\begin{enumerate}
		\item 	Let $\tau:\mathcal{C}\to \mathcal{C}$ be a fixed map, such that for each $[\gamma]\in \mathcal{C}$, the algebraic intersection number of $[\gamma]$ and $\tau([\gamma])$ is $\pm 1$.
		\item For each $[\gamma]\in \mathcal{C}$, let $\mathfrak{o}([\gamma])$ be a fixed choice of orientation of $[\gamma]$. 
	\end{enumerate} 
\end{Definition}

Definition \ref{def_tau_orientation} allows us to define the following two canonical maps. Recall that we use $\mathcal{U}_1\subset (-1,1)\times A$ to denote the trivial annular knot and use $\mathcal{K}_1\subset (-1,1)\times A$ to denote the annular knot given by the closure of a 1-braid. In the following, we also require that $\mathcal{U}_1$ and $\mathcal{K}_1$ are both included in $\{0\}\times A$. 
	
\begin{Definition}
\label{def_Phi}
Let $c_0\subset \Sigma $ be a simple closed curve, and let $K=\{0\}\times c_0$. 
\begin{enumerate}
	\item If $c_0$ is an essential curve, define 
$$\Phi:\SiHI(K)\to \AHI(\mathcal{K}_1)$$ 
to be the map $\Phi_{c,c'}$ given by \eqref{eqn_Phi_c,c'_essential} such that $c$ is parallel to $c_0$, and $c'$ is in the isotopy class of $\tau([c])$.

	\item If $c_0$ is a trivial curve, define 
 $$\Phi:\SiHI(K)\to \AHI(\mathcal{U}_1)$$ 
  to be the map $\Phi_{c,c'}$ given by \eqref{eqn_Phi_c,c'_essential} such that $(c, c')$ is an arbitrary pair of curves that intersect positively at one point and $c_0\subset N(c)$.
\end{enumerate}
\end{Definition} 

\begin{remark}
The map $\Phi$ is well-defined up to a multiplication by $\pm 1$. Lemma \ref{lem_naturality_v_+-} below will show that the map $\Phi$ in Part (2) of the definition does not depend on the choice of $c$ and $c'$.
\end{remark}

Recall that by Definition \ref{def_generator_u_0} and Remark \ref{rmk_AHI_emptyset}, we have chosen a fixed generator  $\bfu_0$ of $\AHI(\emptyset)\cong \bC$. 

Recall that $A$ denotes an annulus. Let $\mathcal{U}_1$ be an (oriented) unknot in $(-1,1)\times A$.
Let $E^+$  be an (oriented) disk  in 
$[-1,1]\times S^1\times D^2$ that gives a cobordism from the empty link to $\mathcal{U}_1$. Let $E^-$ be obtained from $E^+$ by taking connected sum with a torus that is contained in a small solid ball. Let $\bfv_+,\bfv_-\in \AHI(\mathcal{U}_1)$ be defined by
$$
\bfv_+=\AHI(E^+)(\bfu_0),~ \bfv_-=\frac12 \AHI(E^-)(\bfu_0),
$$
where the signs of $\AHI(E^\pm)$ are given as in Remark \ref{rmk_sign_link_cobordism}. 
By \cite[Proposition 5.9]{AHI}, we have 
\begin{equation}
	\label{eqn_basis_AHI_U_1}
	\AHI(\mathcal{U}_1)=\bC\bfv_+\oplus \bC\bfv_-.
\end{equation}

Since $\Sigma$ is a torus, we have
\begin{equation*}
\bC\cong \SiHI(\emptyset)=\II(S^1\times \Sigma,\emptyset, S^1\times \{p\}|\Sigma)   =
\II(S^1\times \Sigma,\emptyset, S^1\times \{p\})_{(2)}.
\end{equation*}
The representation variety of the admissible triple
$
(S^1\times \Sigma,\emptyset, S^1\times \{p\})
$
consists of two regular points differ by degree $4$. Hence there is no differential in the Floer chain complex and
$
\II(S^1\times \Sigma,\emptyset, S^1\times \{p\})
$
is freely generated  by the two points.
Denote the critical points by $\rho$, $\rho'$. By abuse of notation, define 
$$
\bfu_0=\frac12(\rho+\frac12\mu(\pt)(\rho))\in \SiHI(\emptyset)
$$ 
as a generator of   
$\SiHI(\emptyset)$.

\begin{Lemma}\label{lem_naturality_v0_u0}
Let 
$$
\Phi_{c,c'}: \SiHI(\emptyset)\to \AHI(\emptyset)
$$ 
be the map given by \eqref{eqn_Phi_c,c'_essential}.
We have
$$
\Phi_{c,c'}(\bfu_0)=\pm \bfu_0
$$
for every choice of $c,c'$.
\end{Lemma}
\begin{proof}
The map $\Phi_{c,c'}$ is induced by the excision cobordism from
$$
(S^1\times T^2,\emptyset, S^1\times \{\pt\})\sqcup (S^1\times S^2,S^1\times \{q_1,q_2\},u)
$$
to
$$
(S^1\times S^2,S^1\times \{q_1,q_2\},u),
$$
where $u$ is an arc on $\{\pt\}\times S^2$ connecting $S^1\times \{q_1\}$ and $S^1\times \{q_2\}$.
This cobordism is obtained by attaching
$$
T^2\times [-1,1]^2
$$
to
$$
[-1,1]\times ((S^1\times T^2,\emptyset,\{q\}\times S^1)\sqcup (S^1\times S^2,S^1\times \{q_1,q_2\},u)),
$$ 
and identifying
$$
T^2\times  [-1,1]\times \{+1\}
$$
with  
$$N(\{1\}\times T)\subset \{1\}\times S^1\times \Sigma,$$
 and identifying
$$
T^2\times  [-1,1]\times \{-1\}
$$
with
$$
N(\{1\}\times\hat{T})\subset \{1\}\times S^1\times S^2.
$$
A direct calculation shows that the moduli space of flat connections on this 
cobordism consists of two points whose restriction to $S^1\times T^2$
are $\rho$ and $\rho'$ respectively. Therefore the cobordism map takes the elements
$$[\rho]\otimes \bfu_0, [\rho']\otimes \bfu_0\in (S^1\times T^2,\emptyset, S^1\times \{\pt\})\sqcup (S^1\times S^2,S^1\times \{q_1,q_2\},u)$$ to 
$$\pm \bfu_0\in (S^1\times S^2,S^1\times \{q_1,q_2\},u).
$$
Hence we have $\Phi_{c,c'}(\bfu_0)=\pm \bfu_0$.
\end{proof}

Let $U_1$ be an (oriented) unknot in $(-1,1)\times \Sigma$. By abuse of notation, let $E^+$ be an (oriented) disk in $[-1,1]\times (-1,1)\times \Sigma$ that gives a cobordism from the empty link to $U_1$, and let $E^-$ be obtained from $E^+$ by taking connected sum with a torus that is contained in a small solid ball. Define $\bfv_+,\bfv_-\in \SiHI(U_1)$ by 
$$
\bfv_+=\SiHI(E^+)(\bfu_0),~ \bfv_-=\frac12 \SiHI(E^-)(\bfu_0),
$$
where the signs of $\bfv_\pm$ are fixed as in Remark \ref{rmk_sign_link_cobordism}. 
Lemma \ref{lem_naturality_v0_u0} and the functoriality of instanton Floer homology yield 
the following result.
\begin{Lemma}\label{lem_naturality_v_+-}
Let $c_1$ be an essential simple closed curve on $\Sigma$, let $N(c)\subset \Sigma$ be a tubular neighborhood of $c$, let $U_1\subset \{0\}\times N(c)$ be an unknot. Let $c'$ be a simple closed curve on $\Sigma$ that intersects $c$ positively at one point.
Then there exists 
$h\in \{1,-1\}$, such that
$$
 \Phi_{c,c'}(\bfv_\pm)=h \bfv_\pm,
$$
where $\Phi_{c,c'}:\SiHI(U_1)\to \AHI(\mathcal{U}_1)$ is the map given by \eqref{eqn_Phi_c,c'_essential}. \qed
\end{Lemma}
As a consequence, we have
\begin{equation}
\label{eqn_basis_SiHI_U_1}
	\SiHI(U_1)=\bC\bfv_+\oplus \bC\bfv_-.
\end{equation}

Now let $L\subset (-1,1)\times \Sigma$ be an unlink with $n$ components $K_1,\cdots,K_n$. By Lemma \ref{lem_disjoint_union_of_links} and Lemma \ref{lem_surface_knot_rank_2}, we have $\SiHI(L)\cong \bC^{2^n}$. We construct a canonical basis for $\SiHI(L)$ as follows. 
Let $\{v_1^+,\cdots,v_n^+,v_1^-,\cdots,v_n^-\}$ be a set of $2n$ elements. For each $i=1,\cdots,n$, let $V_i$ be the 2-dimensional $\bC$--linear space generated by $v_i^+$ and $v_i^-$. Define the map
\begin{equation*}
	\Theta_L: \otimes_{i=1}^n V_i \to\SiHI(L).
\end{equation*}
by 
\begin{equation}
\label{eqn_Theta_L}
\Theta_L(v_1^\pm \otimes \cdots \otimes v_n^\pm)=\SiHI(E_1^\pm\sqcup \cdots \sqcup E_n^\pm)(\bfu_0),
\end{equation}
where $E_i^+$ are (oriented) disk cobordisms from the empty set to $K_i$, and $E_i^-$ is obtained from $E^+_i$ by taking connected sum with a torus that is contained in a small solid ball. Notice that the map $\phi$ given by Lemma \ref{lem_disjoint_union_of_links} is natural with respect to the maps induced by ``split'' cobordisms. By Equation \eqref{eqn_basis_SiHI_U_1} and Lemma \ref{lem_disjoint_union_of_links}, the map $\Theta_L$ is an isomorphism.

\begin{Lemma}
\label{lem_Theta_concordance}
The isomorphism $\Theta_L$ does not depend on the choices
of $E_i^{\pm}$.
Moreover, let $L_0=K_1^{(0)}\cup\cdots\cup K_n^{(0)}$ and $L_1=K_1^{(1)}\cup\cdots\cup K_n^{(1)}$ be two unlinks in $(-1,1)\times \Sigma$, and let $S$ be an (oriented) link concordance from $L_1$ to $L_2$ consisting of the disjoint union of concordances from $K_i^{(0)}$ to $K_i^{(1)}$ for all $i$, then we have
$$
\SiHI(S)\circ\Theta_{L_0}  = \Theta_{L_1}.
$$
\end{Lemma}

\begin{proof}
The lemma is an immediate consequence of Proposition \ref{prop_homotopy_invariance_inst} below.
\end{proof}

\begin{Proposition}\label{prop_homotopy_invariance_inst}
Let $n_0,n_1,m_0,m_1$ be non-negative integers.
For $i=0,1$, suppose $K_{m_i}$ consists of $m_i$ parallel copies of essential simple closed curves in 
$\{0\}\times \Sigma$,
and $U_{n_i}$ is an unlink with $n_i$ components in  $(-1,1)\times \Sigma - K_{m_i}$.  
Given two oriented cobordisms $S,S'\subset [-,1,1]\times (-1,1)\times\Sigma$ from 
$U_{n_0}\cup K_{m_0}$ to $U_{n_1}\cup K_{m_1}$. If $S$ and $S'$ are \emph{homotopic} relative to the boundary, then 
$$
\SiHI(S)=\SiHI(S').
$$
\end{Proposition}
\begin{proof}
We need to use the local system introduced in \cite{KM:YAFT}. Let $L\subset (-1,1)\times\Sigma$ be a link, let $\mathcal{R}$ be the ring 
$\mathbb{C}[t,t^{-1}]$ and $\mathcal{B}$ be the configuration space of orbifold connections on 
$$(S^1\times \Sigma,L,S^1\times \{p\}).$$
Given any continuous function 
$$
s:\mathcal{B}\to S^1,
$$
a local system of free rank-1 $\mathcal{R}$-modules on $\mathcal{B}$ is defined in  \cite{KM:YAFT}*{Section 3.9}.
Given any orbifold connection $[A]\in \mathcal{B}$, its holonomy along any component of $L$ lies in a $S^1$-subgroup of $SU(2)$.
In particular, we can take $s$ to be the product of holonomies along all the components of $L$. In this way we obtain a 
local system $\Gamma$ and the corresponding Floer homology group
$$
\SiHI(L;\Gamma)
$$
which is the homology of a chain complex of free $\mathcal{R}$-modules generated by the  critical points of the Chern-Simons functional.
A similar local system can also be defined for the annular instanton Floer homology $\AHI$. Moreover,
Lemma \ref{lem_surface_reduce_to_AHI} still holds when we use these local systems. 

Now take $L_i=U_{n_i}\sqcup K_{m_i}$ ($i=0,1$). We have
$$
\SiHI(L_i;\Gamma)\cong \AHI(L_i;\Gamma)\cong \mathcal{R}^{2^{n_i+m_i}}
$$
where the second isomorphism is from \cite[Example 3.4]{XZ:forest}.
Therefore $\SiHI(L_i;\Gamma)$ ($i=1,2$) is a free $\mathcal{R}$-module. 

We can view $\mathbb{C}$ as the $\mathcal{R}$-module $\mathcal{R}/(t-1)$. In order to show that
$$
\SiHI(S;\mathbb{C})=\SiHI(S';\mathbb{C})
$$
it suffices to show 
\begin{equation}\label{AHI=Gamma}
\SiHI(S;\Gamma)=\SiHI(S';\Gamma)
\end{equation}
by the universal coefficient theorem. Since $S$ and $S'$ are homotopic, $S'$ can be obtained from $S$ by a sequence of the following moves:
\begin{enumerate}
	\item an ambient isotopy of $[-1,1]\times(-1,1)\times\Sigma$,
	\item a twist move that introduces a positive double point, 
	\item a twist move that introduces a negative double point,
	\item a finger move that introduces two double points with opposite signs,
\end{enumerate}
or their inverses \cite{freedman2014topology}. Since both $S$ and $S'$ are embedded, they have the same number of  positive double points (namely zero), therefore the signed count of moves (2) and (4) is zero. By \cite{KM:YAFT}*{Proposition 5.2}, there exists $m\in \mathbb{N}$ such that
\begin{equation}\label{AHI=Gamma=t}
(t^{-1}-t)^m\SiHI(S;\Gamma)=(t^{-1}-t)^m \SiHI(S';\Gamma).
\end{equation}
Since $\SiHI(L_i;\Gamma)$ are free $\mathcal{R}$-modules, \eqref{AHI=Gamma} follows from \eqref{AHI=Gamma=t} 
\end{proof}

Now suppose $L$ is a link embedded in $\{0\}\times \Sigma$, and let $\gamma=\gamma_1\cup\cdots\cup\gamma_m$ be the disjoint union of $m$ parallel essential simple closed curves on $\Sigma$ such that  $\{0\}\times \gamma$ is disjoint from $L$.
Suppose $\gamma$ splits $L$ into the disjoint union of (possibly empty) links $L_1,\cdots,L_m$. Isotope $\Sigma$ such that $\gamma$ is linearly embedded. Let $c$ be an oriented curve that is parallel to $\gamma_i$, where the orientation of $c$ is given by $\mathfrak{o}$ in Definition \ref{def_tau_orientation}. Take a family of $m$ linearly embedded tori $T_1,\cdots,T_{m}\subset S^1\times \Sigma$ with fundamental class  $[S^1\times c] +[\Sigma]$, such that $T_i\cap(\{0\}\times \Sigma) = \gamma_i$. We define
	\begin{equation}
	\label{eqn_def_Psi_gamma}
		\Psi_\gamma : \SiHI(L)\to \otimes_{i=1}^m \SiHI(L_i)
	\end{equation}
to be the isomorphism given by the excision cobordism of cutting $S^1\times \Sigma$ open along $T_1,\cdots, T_{m}$ and gluing the boundary using vertical translation maps. The map $\Psi_\gamma$ is well-defined up to a sign. Moreover, $\Psi_\gamma$ is natural with respect to the maps induced by ``split'' cobordisms.

\begin{remark}
	Let $L$ and $L_1,\cdots,L_m$ be as above, then Lemma \ref{lem_disjoint_union_of_links} already proved 
	\begin{equation*}
			\SiHI(L)\cong \otimes_{i=1}^m \SiHI(L_i).
	\end{equation*}
	The purpose of  \eqref{eqn_def_Psi_gamma} is to choose a canonical map that realizes this isomorphism. 
\end{remark}

Let $c$ be a non-separating simple closed curve on $\Sigma$, we use $K_{m}$ to 
denote $m$ parallel copies of $c$ in 
$\{0\}\times \Sigma$. 
Let $\gamma$ be a union of $m$ parallel essential curves on $\Sigma$ such that $\{0\}\times \gamma$ is disjoint from $K_m$ and it splits $K_m$ 
into the disjoint union of  $m$ copies of $K_1$. 
By \eqref{eqn_def_Psi_gamma}, we have the following isomorphism
\begin{equation}
\label{eqn_Psi_gamma_K_m}
		\Psi_\gamma: \SiHI(K_m)\to \SiHI(K_1)^{\otimes m}.
\end{equation}

Suppose $U_n$ is an $n$-component unlink in $(-1,1)\times \Sigma-K_m$, and fix an order of the components of $U_{n}$.
We  extend the definition of \eqref{eqn_Theta_L} to the link $U_n\cup K_m$. Let $L=U_n\cup K_m\subset (-1,1)\times \Sigma$, 
let $V_1,\cdots,V_n$ be as in \eqref{eqn_Theta_L}. First define the map
\begin{equation*}
 F:	\bigotimes_{i=1}^n V_i \otimes \SiHI(K_l)  \to \SiHI(L)
\end{equation*}
by 
$$
F(v_1^\pm\otimes \cdots \otimes v_n^\pm\otimes x):=
\SiHI(E_1^\pm\sqcup \cdots \sqcup E_n^\pm \sqcup [-1,1]\times K_m)(x)
$$
where $x\in \SiHI(K_m)$, and $v_i^\pm\in V_i$, $E_i^\pm\subset [-1,1]\times(-1,1)\times\Sigma-[-1,1]\times K_m$ are the same as in \eqref{eqn_Theta_L}.
Now define
$$
\Theta_L:	\bigotimes_{i=1}^k V_i \otimes \SiHI(K_1)^{\otimes l}  \to \SiHI(L)
$$
by 
\begin{equation}
\label{eqn_def_Theta_L}
	\Theta_L=F\circ \big(\id_{\bigotimes_{i=1}^k V_i}\otimes \Psi_\gamma^{-1}\big),
\end{equation}
where $\Psi_\gamma$ is the map given by \eqref{eqn_Psi_gamma_K_m}.
Since $\Psi_\gamma$ is an isomorphism, it follows from Equation \eqref{eqn_basis_SiHI_U_1} and Lemma \ref{lem_disjoint_union_of_links} that $\Theta_L$ is an isomorphism.

\begin{Lemma}\label{lem_canonical_decomposition_SiHI}
The isomorphism $\Theta_L$ in \eqref{eqn_def_Theta_L} does not
depend on the choices of $E_i^\pm$.
Let $K_m$ be the link as above and $U_n, U_n'$ be two $n$-component unlinks in
$(-1,1)\times \Sigma- K_m$. Denote $U_n\cup K_m$ and $U_n'\cup K_m$ by
$L$ and $L'$ respectively.
Fix an order of the components of $U_n$ and $U_n'$.
Let $S\subset [-1,1]\times(-1,1)\times \Sigma$ be a concordance from $L$ to $L'$ that is the product concordance on $K_l$ and preserves the orders of the components of $U_{k}$ and $U_k'$. 
Let $\Theta_L, \Theta_{L'}$ be the maps given by \eqref{eqn_def_Theta_L}.
Then we have  
$$
\SiHI(S)\circ \Theta_{L}=
\pm \Theta_{L'}.
$$
\end{Lemma} 
\begin{proof}
The lemma is an immediate consequence of Proposition \ref{prop_homotopy_invariance_inst}.
\end{proof}

\begin{remark}
The sign ambiguity of Lemma \ref{lem_canonical_decomposition_SiHI} comes from the fact that $\Psi_\gamma$ is only well-defined up to a sign.
\end{remark}

A similar map can be defined for the annular instanton Floer homology. Recall that $A$ denotes the annulus. We use $\mathcal{K}_m$ to denote the annular link given by the closure of a trivial $m$--braid, and we require that $\mathcal{K}_m$ is included in $\{0\}\times A$. Let $\mathcal{U}_n$ be an unlink in  $(-1,1)\times A-\mathcal{K}_m$, and let $L=\mathcal{U}_n\cup \mathcal{K}_m$. 

Fix an order of the components of $\mathcal{U}_n$.
Let $\{V_i\}_{i=1}^n$ be 2-dimensional complex vector spaces associated to the components of $\mathcal{U}_n$ as in \eqref{eqn_Theta_L},
then the same construction as above gives an isomorphism 
$$
F^{\AHI}:\otimes_{i=1}^n {V_i}\otimes 
\AHI(\mathcal{K}_m)\to \AHI(L),
$$
By \cite[Proposition 4.3]{AHI}, there is an isomorphism
$$
	\Psi^{\AHI}:\AHI(\mathcal{K}_m)\to \AHI(\mathcal{K}_1)^{\otimes m},
$$
and we define
$$
	\Theta_L^{\AHI}: \otimes_{i=1}^n {V_i}\otimes 
\AHI(\mathcal{K}_1)^{\otimes m} \to \AHI(L)
$$
by 
\begin{equation}
\label{eqn_def_Theta_AHI}
	\Theta_L^{\AHI}=F\circ \big(\id_{\bigotimes_{i=1}^k V_i}\otimes (\Psi^{\AHI})^{-1}\big).
\end{equation}

\section{The spectral sequence}
\label{sec_spectral_sequence}

Throughout this section, we will use $\Sigma$ to denote an (oriented) 2-torus.

Given a link $L\subset (-1,1)\times \Sigma$, the Khovanov skein homology
 $\SiKh(L)$ is graded by $\bZ\mathcal{C}$, where $\mathcal{C}$ denotes the isotopy classes of 
essential simple closed curves on $\Sigma$. In order to establish a connection between
$\SiKh(L)$ and $\SiHI(L)$,
we reduce the grading over $\bZ\mathcal{C}$ to a coarser grading system. 
Orient each element of $\mathcal{C}$ using the orientation system $\mathfrak{o}$ given by  Definition \ref{def_tau_orientation}. Let $c\subset \Sigma$ be an oriented essential 
simple closed curve. Recall that $\gamma\cdot c$ denotes the algebraic intersection number. We define a group homomorphism
\begin{align*}
    \bZ\mathcal{C} &\to \bZ \\
    [\gamma]    &\mapsto  \gamma \cdot c.
\end{align*}
This map reduces the $\bZ\mathcal{C}$--grading to a $\bZ$--grading on $\CKh(L)$ and $\SiKh(L)$, which will be called the \emph{$c$--grading}.
By definition, the $c$--grading of the generator $v(\gamma)_\pm$ is given by $\pm \gamma\cdot c$. Let $\SiKh_c(L,i;\bZ/2)$ and $\SiHI_c(L,i)$ be the components of $\SiKh(L;\bZ/2)$ and $\SiHI(L)$ respectively with 
$c$--grading equal to $i$.

\begin{Theorem}\label{theorem_spectral_sequence} 
Suppose $L\subset (-1,1)\times \Sigma$ is a link and 
$c\subset \Sigma$ is an oriented essential simple closed curve.
Then
$$
\rank_{\bZ/2} \SiKh_c(L,i;\bZ/2)\ge \dim_\bC \SiHI_c(L,i).
$$ 
\end{Theorem}

We can now prove the main results of Section \ref{sec_intro} assuming Theorem \ref{theorem_spectral_sequence}. 
\begin{reptheorem}{thm_disjoint_from_annulus_intro}
Suppose $L$ is a link in $(-1,1)\times \Sigma$ and $c$ is an oriented essential simple closed curve on $\Sigma$.
Then L can be isotoped to a link disjoint from $(-1,1)\times c$
 if and only if $\SiHI(L;\bZ/2)$ is supported at $c$--degree $0$. 
\end{reptheorem}

\begin{proof}
The ``only if'' part is trivial, we only need to prove the ``if'' part.
Suppose  $\SiHI(L;\bZ/2)$ is supported at $c$--degree $0$, then by 
Theorem \ref{theorem_spectral_sequence}, $\SiHI(L)$ is supported at $c$--degree $0$, and the result follows from Proposition \ref{prop_c_grading_zero_implies_disjoint}.
\end{proof}

\begin{reptheorem}{thm_main_detection_result}
Suppose $L$ is a non-empty link in $(-1,1)\times \Sigma$ such that 
$$\rank_{\bZ/2} \SiKh(L;\bZ/2)\le 2.$$ Then $L$ is isotopic to 
a knot embedded in $\{0\}\times \Sigma$, and hence 
$$\rank_{\bZ/2} \SiKh(L;\bZ/2)= 2.$$
\end{reptheorem}
\begin{proof}
By Theorem \ref{theorem_spectral_sequence} and Proposition \ref{prop_inst_rank_2_disjoint_from_c}, there
is an essential simple closed curve $c\subset \Sigma$ such that $L$ can be isotoped to 
a link disjoint from $(-1,1)\times c$. Therefore $L$ can be viewed as a link in 
$(-1,1)\times (\Sigma- N(c))\cong (-1,1)\times A$, where $A$ denotes an annulus. 

By definition, $\rank_{\bZ/2}\SiKh(L;\bZ/2)= \rank_{\bZ/2}\AKh(L;\bZ/2)$.
On the other hand $\AKh(L;\bC)$ carries an $\mathfrak{sl}_2(\bC)$-action 
(see \cite{GLW_AKh_rep})
where the $f$-grading is 
the weight. Since
$$
\dim_\bC \AKh(L;\bC)\le \rank_{\bZ/2} \AKh(L;\bZ/2)=\rank_{\bZ/2}\SiKh(L;\bZ/2)\le 2,
$$
this implies the top $f$-grading of $\AKh(L;\bC)$ is no greater than $1$.
Therefore by \cite[Theorem 1.3]{XZ:excision}, $L$ is either a parallel copy of $c$, or is contained in a solid 3--ball. In the latter case, $\SiKh(L;\bZ/2)$ has the same rank as the Khovanov homology of $L$ as a link in the 3-ball, therefore it follows from Kronheimer-Mrowka's unknot detection theorem \cite{KM:Kh-unknot} that $L$ is an unknot. Either way, $L$ is isotopic to a knot embedded in $\{0\}\times\Sigma$.
\end{proof}
The rest of this section is devoted to the proof of Theorem \ref{theorem_spectral_sequence}. 

\subsection{Construction of the spectral sequence}
Let $L$ be a link in $(-1,1)\times\Sigma$ and $D$ be a diagram of $L$ with $d$ crossings. 
Given any $v\in \{0,1\}^d$, let $L_v\subset \{0\} \times \Sigma\subset (-1,1)\times\Sigma$
be the corresponding resolved diagram.

Let $(C_v,d_v)$ be the Floer \emph{chain complex} of 
$$
\SiHI(L_v)=\II(S^1\times \Sigma,L_v,S^1\times \{p\}).
$$
Let $C(L)$ be the Floer chain complex of $\SiHI(L)$. 

For $v\in \{0,1\}^d$, let $\|v\|_1$ be the sum of coordinates of $v$.
If $u,v\in\{0,1\}^d$, we write $v\ge u$ if all the coordinates of $v$ are greater than or equal to the corresponding coordinates of $u$.
Suppose $v\ge u$, then there is a standard cobordism $S_{vu}\subset [-1,1]\times (-1,1)\times \Sigma$
from $L_v$ to $L_u$.

By \cite[Section 6]{KM:Kh-unknot}, we have the following result.
\begin{Proposition}[{\cite[Theorem 6.8]{KM:Kh-unknot}}]
\label{prop_KM_spectral_seq}
There exist linear maps $f_{vu}:C_v\to C_u$ for all $u,v\in \{0,1\}^d$ with $v\ge u$, such that
$$
(\mathbf{C},\mathbf{D}):=(\bigoplus_{v\in \{0,1\}^d} C_v, \sum_{v\ge u} f_{vu})
$$
is a chain complex, and for $v=u$, we have $f_{vv}=d_v$, for $v\ge u$ and $\|v-u\|_1=1$, we have
$f_{vu}=\pm \SiHI(S_{vu})$.
Moreover, there is a chain map
$$
\mathbf{H}:C(L)\to \mathbf{C}
$$
which induces an isomorphism on homology.
\end{Proposition}

Let $c\subset\Sigma$ be an oriented essential simple closed curve. 
Let $r_v:C_v\to C_v$ be the chain map which defines the action
$\muu(S^1\times c):\SiHI(L_v)\to \SiHI(L_v)$.  
The discussion in \cite[Section 5]{AHI} yields the following result.
\begin{Proposition}[{\cite[Proposition 5.7]{AHI}}]
There exists linear maps $r_{vu}:C_v\to C_u$ for all $v\ge u$ in $\{0,1\}^d$ such that
$$
\mathbf{R}=\sum_{v\ge u} r_{vu}:\mathbf{C}\to \mathbf{C}
$$
is a chain map with respect to the differential $\mathbf{D}$ in Proposition \ref{prop_KM_spectral_seq}, and $r_{vv}=r_v$. Moreover,
the isomorphism
$$
\mathbf{H}_\ast: \SiHI(L)\to H_\ast(\mathbf{C})
$$
intertwines $\muu(S^1\times c)$ on  $\SiHI(L)$ with $\mathbf{R}_\ast$ on 
$ H_\ast(\mathbf{C})$.
\end{Proposition}

By filtering the cube $\mathbf{C}$ with the $l^1$-norm on $\{0,1\}^d$, 
the above two propositions give  
 the following spectral sequence. We use $\SiHI_c(L,i)$ to denote the component of  $\SiHI(L)$ with $c$--grading equal to  $i$.
\begin{Lemma}\label{lemma_spectral_sequence_E1}
Let $c$ be an essential simple closed curve on $\Sigma$ and let $i\in\bZ$. There exists a spectral sequence which converges to $\SiHI_c(L,i)$ and whose $E_1$--page is given by 
$$
(E_1,d_1)=\big(\bigoplus_{v\in \{0,1\}^d} \SiHI_c(L_v,i);\sum_{\substack{ v\ge u \\ \|v-u\|_1=1 }} \pm \SiHI(S_{vu})\big).
$$
\end{Lemma}

\subsection{Differentials on the $E_1$--page}
In order to find the $E_2$--page of the spectral sequence given by Lemma \ref{lemma_spectral_sequence_E1}, we need to calculate the differentials on the 
$E_1$--page. Our strategy is to use the isomorphisms in Section \ref{sec_canonical_isomorphisms} to reduce the computation to a known result on the annular instanton Floer homology from \cite{AHI}.

\begin{Proposition}\label{prop_mod_2_naturality_excision_cobordism}
Let $c$ be an essential simple closed curve on $\Sigma$, let $N(c)\subset \Sigma$ be a tubular neighborhood of $c$, let $L$ be a link embedded in $\{0\}\times N(c)\subset (-1,1)\times \Sigma$. Suppose $L$ has $n$ contractible components and $m$ essential components. Fix an order of the contractible components of $L$.

Let $c'\subset\Sigma$ be an essential simple closed curve isotopic to $\tau([c])$, where $\tau$ is the map given by Definition \ref{def_tau_orientation}.
 Let $\Theta_L^{\AHI}$ be the map given by \eqref{eqn_def_Theta_AHI} when viewing $L$ as an annular link in $\{0\}\times N(c)$, where $N(c)$ is identified with the standard annulus using the framing given by $c'$.
  
Let $\Phi_{c,c'},\Phi,\Theta_L$ be the isomorphisms defined in Section \ref{sec_canonical_isomorphisms}.
We have
\begin{equation}
\label{eqn_natruality_SiHI_AHI_Phi_Theta}
	   \Phi_{c,c'}\circ \Theta_L =\pm \Theta_L^{\AHI} \circ \Big(\id_{\otimes_{i=1}^n V_i}\otimes \big(\Phi^{\otimes m}\big)\Big) 
\end{equation}
as maps from $\big(\otimes_{i=1}^n V_i\big)\otimes \SiHI(\{0\}\times c)^{\otimes m}$ to $\AHI(L)$.
\end{Proposition}
\begin{proof}
The compositions of the cobordisms which induce the maps on the two sides of \eqref{eqn_natruality_SiHI_AHI_Phi_Theta} are the diffeomorphic, therefore the result follows 
from the functoriality of instanton Floer homology. 
\end{proof}

For a given essential simple closed curve $c\subset \Sigma$, recall that we use $K_m\subset \{0\}\times \Sigma$ to denote the disjoint union of $m$ parallel copies of $c$, and we view $K_m$ as a link in $(-1,1)\times \Sigma$. We also use $U_n$ to denote an $n$--component unlink in $(-1,1)\times \Sigma-K_m$.

We introduce six (oriented) cobordisms for links in $(-1,1)\times \Sigma$.
\begin{itemize}
\item 
Let $S_1\subset [-1,1]\times (-1,1)\times \Sigma$ be the standard pair-of-pants cobordism from $U_2$ to $U_1$,
and $\bar{S}_1$ be the standard pair-of-pants cobordism from $U_1$ to $U_2$. 

\item
Let $S_2\subset [-1,1]\times (-1,1)\times \Sigma$ be the standard pair-of-pants cobordism from $U_1\sqcup K_1$ to 
$K_1$,
and $\bar{S}_2$ be the standard pair-of-pants cobordism from $K_1$ to $U_1\sqcup K_1$.

\item 
Let $S_3\subset [-1,1]\times (-1,1)\times \Sigma$ be the standard pair-of-pants cobordism from  $K_2$ to 
$U_1$
and $\bar{S}_3$ be the standard pair-of-pants cobordism from $U_1$ to $K_2$. 

\end{itemize}
Notice that in the last two cobordisms, the two parallel components of $K_2$ are oriented oppositely.
All these cobordisms can be included in $[-1,1]\times (-1,1)\times N(K_1)\cong [-1,1]\times S^1\times D^2$. Moreover, the signs of the cobordism maps induced by $S_i$ and $\bar{S}_i$ can be fixed as in Remark \ref{rmk_sign_link_cobordism}.

Recall that $\mathcal{U}_1$ denotes the trivial annular knot, and $\mathcal{K}_1$ denotes the annular knot given by the closure of a 1-braid embedded in $\{0\}\times A$.
In the next proposition, we will view $U_n\cup K_m$ as an annular link in $(-1,1)\times N(c)$, and use the isomorphism $\Theta^{\AHI}$ to identify $\AHI(U_n\cup K_m)$ with $\AHI(\mathcal{U}_1)^n\otimes \AHI(\mathcal{K}_1)^m$. 

Also recall that by  \eqref{eqn_basis_AHI_U_1}, we have  
$$
\AHI(\mathcal{U}_1)=\bC \bfv_+\oplus \bC \bfv_-.
$$

\begin{Proposition}[{\cite[Proposition 5.9, Proposition 5.14]{AHI}}]\label{prop_AHI_differential}
For an appropriate choice of the sign of $\Theta^{\AHI}$, we have
\begin{align*}
&\AHI(S_1)(\bfv_+\otimes \bfv_\pm)=\bfv_\pm,  \\
&\AHI(S_1)(\bfv_\pm\otimes \bfv_+)=\bfv_\pm,  \\
&\AHI(S_1)(\bfv_\pm\otimes \bfv_-)=0,  \\
&\AHI(\bar{S}_1)(\bfv_+)       =\bfv_+\otimes \bfv_- +  \bfv_-\otimes \bfv_+, \\
&\AHI(\bar{S}_1)(\bfv_-)       =\bfv_-\otimes \bfv_-.
\end{align*}
Moreover,
there are generators $\bfw_\pm$ of $\AHI(\mathcal{K}_1)$ with f-degree $\pm 1$ such that
\begin{align*}
&\AHI(S_2)(\bfv_+\otimes \bfw_\pm)=\bfw_\pm,  \\
&\AHI(S_2)(\bfv_-\otimes \bfw_\pm)=0,  \\
&\AHI(\bar{S}_2)(\bfw_\pm)       =\bfv_-\otimes \bfw_\pm, \\ 
&\AHI({S}_3)(\bfw_+\otimes \bfw_+)=0,              \\
&\AHI({S}_3)(\bfw_-\otimes \bfw_-)=0,       \\
&\AHI({S}_3)(\bfw_+\otimes \bfw_-)=\bfv_-,   \\
&\AHI({S}_3)(\bfw_-\otimes \bfw_+)=\bfv_-,    \\
&\AHI(\bar{S}_3)(\bfv_+)=\bfw_+\otimes \bfw_- + \bfw_-\otimes \bfw_+, \\
&\AHI(\bar{S}_3)(\bfv_-)=0.
\end{align*}
\end{Proposition}

Let $\bfw_\pm$ be the generators of $\AHI(\mathcal{K}_1)$ given by Proposition \ref{prop_AHI_differential}.
Using the isomorphism $\Phi$ from Definition \ref{def_Phi}, we obtain generators $\Phi^{-1}(\bfw_\pm)$ of $\SiHI(K_1)$. By abuse of notation, we will use $\bfw_\pm$ to denote $\Phi^{-1}(\bfw_\pm)$ when there is no risk of confusion.

Suppose $c'$ is isotopic to $\tau([c])$ where $\tau$ is the map given by Definition \ref{def_tau_orientation}, and orient $c$ using the orientation system $\mathfrak{o}$ in Definition \ref{def_tau_orientation}, orient $c'$ such that $c\cdot c'=1$. Then the $c'$--grading of $\bfw_\pm$ is given by $\pm 1$. 

In the next proposition, we use the map $\Theta_L$ 
to identify $\SiHI(U_n\cup K_m)$ with $\SiHI(U_1)^{\otimes n}\otimes \SiHI(K_1)^{\otimes m}$.

\begin{Proposition}\label{prop_differeltial_pair_pants}
For $i=1,2,3$, there exist $h_i,h_i'\in\{1,-1\}$ such that $h_i\SiHI(S_i)$ and $h_i' \SiHI(\bar{S}_i)$ are given by the same formulas for $\AHI(S_i)$ and $\AHI(\bar{S}_i)$ as in
Proposition \ref{prop_AHI_differential}.
\end{Proposition}
\begin{proof}
Let $c$ be an essential curve on $\Sigma$ such that $(-1,1)\times \Sigma$ is contained in $(-1,1)\times N(c)$, let $c'$ be an essential curve isotopic to $\tau([c])$ that intersects $c$ positively at one point.
By the functoriality of instanton Floer homology, we have
$$
\Phi_{c,c'}\circ \SiHI(S_1)=\pm  \AHI(S_1)\circ \Phi_{c,c'}.
$$
Therefore, the conclusion follows from Lemma \ref{lem_naturality_v_+-} and Proposition \ref{prop_mod_2_naturality_excision_cobordism}.
\end{proof}

Now let $L$ be an arbitrary link in $(-1,1)\times \Sigma$ that has a diagram with $d$ crossings. Let $v\in \{0,1\}^d$, and let $L_v\subset\{0\}\times  \Sigma$ be the corresponding resolved diagram of $L$. 
Let $c_v\subset \Sigma$ be an essential simple closed curve that is parallel to all the essential components of $L_v$, then one can define the map $\Theta_{L_v}$ using \eqref{eqn_def_Theta_L} by taking $c=c_v$. The isotopy class of $c_v$ is unique if $L_v$ contains at least one essential component.

\begin{Lemma}\label{lemma_spectral_sequence_d1}
For each $v\in \{0,1\}^d$, identify $\SiHI(L_v)$ with $CKh_v(\bar{L})$ using the map $\Theta_{L_v}$ and the decomposition $\SiHI(K_1)=\bC\bfw_{+}\oplus \bC\bfw_-$.
Then up to signs, the differentials on the $E_1$--page of the spectral sequence in Lemma \ref{lemma_spectral_sequence_E1}
coincide with the differentials
of the component of $CKh_v(\bar{L};\bC)$ at $c$--degree $i$.
\end{Lemma}
\begin{proof}
Suppose $v\ge u$ and $\|v-u\|_1$. If $L_v$ contains at least one essential component, then all the essential components of $L_u$ are parallel to $c_v$. If $L_v$ only contains trivial components, then the map $\Theta_{L_v}$ is independent of the choice of $c_v$. Therefore, we may assume without loss of generality that $c=c_u=c_v$ in the definitions of $\Theta_{L_v}$ and $\Theta_{L_u}$.

Suppose $L_v \cong U_{n}\cup K_{m}$ and $L_v  \cong U_{n'}\cup K_{m'}$. For $i\in\{1,2,3\}$, let $S_i$ and $\bar{S}_i$ be the link cobordisms defined above.
Then the cobordism $S_{vu}$ is the union of a copy of $S_i$ or $\bar{S}_i$ and a product cobordism. 
By Proposition \ref{prop_homotopy_invariance_inst} and Lemma \ref{lem_canonical_decomposition_SiHI},  we may assume without loss of generality that there is a tubular neighborhood $N(c)$ of $c$ on $\Sigma$, such that $S-\big([-1,1]\times (-1,1)\times N(c)\big)$ is the trival cobordism, and $S_{uv}\cap [-1,1]\times (-1,1)\times N(c)$ is given by $S_i$ or $\bar{S}_i$. 
The conclusion then follows from 
Proposition \ref{prop_differeltial_pair_pants} and the functoriality of instanton Floer homology.
\end{proof}

\begin{remark}
	It is not clear to the authors whether the signs of the differentials on the $E_1$-page agree with the differentials of $\CKh(\bar{L};\bC)$ after a global conjugation. The answer to this question might require
carefully fixing the signs of the maps defined in Section \ref{sec_canonical_isomorphisms}. Lemma \ref{lemma_spectral_sequence_d1} gives a weaker result that is sufficient for our purpose.
\end{remark}

\begin{proof}[Proof of Theorem \ref{theorem_spectral_sequence}]
Although $(E_1,d_1)$ is defined over $\bC$, Lemma \ref{lemma_spectral_sequence_d1} implies that there is a chain complex $(E_1',d_1')$ over $\bZ$ such that
$$
(E_1,d_1)=(E_1',d_1')\otimes_{\bZ}\bC.
$$
 Let $\beta_{\bZ/2}$ be the rank of the homology of $(E_1',d_1')\otimes_{\bZ}\bZ/2$ over $\bZ/2$, and let $\beta_{\bC}$ be the dimension of the homology of $(E_1',d_1')\otimes_{\bZ}\bC = (E_1,d_1)$ over $\bC$. By the universal coefficient theorem, we have
$$
\beta_{\bZ/2}\ge \beta_{\bC}.
$$
By Lemma \ref{lemma_spectral_sequence_d1}, 
$$
(E_1',d_1')\otimes_{\bZ}\bZ/2\cong \big(CKh_c(\bar{L},i;\bZ/2),d\big),
$$
therefore
$$
\beta_{\bZ/2} = \rank_{\bZ/2} \SiKh_c(\bar{L},i;\bZ/2) = \rank_{\bZ/2} \SiKh_c({L},i;\bZ/2).
$$
The spectral sequence in Lemma \ref{lemma_spectral_sequence_E1} implies
$$
\beta_{\bC} \ge \dim_\bC \SiHI_c(L,i),
$$
and hence the theorem is proved.
\end{proof}

\bibliographystyle{amsalpha}
\bibliography{references}

\end{document}